\documentclass[12pt,reqno]{amsart}
\usepackage[utf8]{inputenc}
\usepackage{amssymb, amsmath, amsthm, url}
\usepackage{hyperref}
\usepackage{todonotes}
\usepackage{mathptmx}
\usepackage{enumerate}
\usepackage[shortlabels]{enumitem}
\usepackage{pgfplots}
\pgfplotsset{compat=1.12}
\usepackage{mathrsfs}
\usetikzlibrary{arrows}

\def\sideremark#1{\ifvmode\leavevmode\fi\vadjust{\vbox to0pt{\vss 
      \hbox to 0pt{\hskip\hsize\hskip1em           
 \vbox{\hsize2cm\tiny\raggedright\pretolerance10000
 \noindent #1\hfill}\hss}\vbox to8pt{\vfil}\vss}}} %
\usepackage{pgf,tikz,pgfplots}
\pgfplotsset{compat=1.12}
\usepackage{mathrsfs}
\usetikzlibrary{arrows}

\newtheorem{introtheorem}{Theorem}

\newtheorem{introcorollary}[introtheorem]{Corollary}
\newtheorem{introproposition}[introtheorem]{Proposition}
\newtheorem{theorem}{Theorem}[section]
\newtheorem{lemma}[theorem]{Lemma}

\newtheorem{proposition}[theorem]{Proposition}
\newtheorem{corollary}[theorem]{Corollary}

\theoremstyle{definition}

\newtheorem{example}[theorem]{Example}

\theoremstyle{remark}
\newtheorem{remark}[theorem]{Remark}

\usepackage{color}

\newcommand{\NE}{\ensuremath {\operatorname {NE}}}
\newcommand{\NS}{\ensuremath {\operatorname {NS}}}
\newcommand{\Nef}{\ensuremath{\operatorname {Nef}}}

\newcommand{\ord}{\ensuremath{\operatorname {ord}}}

\newcommand{\mult}{\ensuremath{\operatorname {mult}}}

\numberwithin{figure}{section}

\begin{document}

\title[The cone of curves and the Cox ring of rational surfaces]{The cone of curves and the Cox ring of rational surfaces over Hirzebruch surfaces}

\author[Galindo]{Carlos Galindo}

\address{Universitat Jaume I, Campus de Riu Sec, Departamento de Matem\'aticas \& Institut Universitari de Matem\`atiques i Aplicacions de Castell\'o, 12071
Caste\-ll\'on de la Plana, Spain} \email{galindo@uji.es}

\author[Monserrat]{Francisco Monserrat}

\address{Universitat Polit\`ecnica de Val\`encia, Departament de Matem\`atica Aplicada \&  Institut Universitari de Matem\`atica Pura i Aplicada, 46022
València, Spain} \email{framonde@mat.upv.es}

\author[Moreno-\'Avila]{Carlos-Jes\'us Moreno-\'Avila}

\address{Universidad de Extremadura, Escuela Politécnica, Departamento de Matemáticas, 10003, Cáceres, Spain} \email{cjmoravi@unex.es}

\thanks{The authors were partially funded by MCIN/AEI/10.13039/501100011033 and by “ERDF A way of
making Europe”, grant PID2022-138906NB-C22, as well as by Universitat Jaume I, grant GACUJIMA-2024-03. The third author was also supported by the Margarita Salas postdoctoral contract MGS/2021/14(UP2021-021) financed by the European Union-NextGenerationEU}

\begin{abstract}
Let $X$ be a rational surface obtained by blowing up at a configuration $\mathcal{C}$ of infinitely near points over a Hirzebruch surface $\mathbb{F}_\delta$. We prove that there exist two positive integers $a \leq b$ such that the cone of curves of $X$ is finite polyhedral and minimally generated when $\delta \geq a$, and the Cox ring of $X$ is finitely generated whenever $\delta \geq b$. The integers $a$ and $b$ depend only on a combinatorial object (a graph decorated with arrows) representing the strict transforms of the exceptional divisors, their intersections and those with the fibers and special section of $\mathbb{F}_\delta$.
\end{abstract}

\maketitle
\section{Introduction}\label{sec:intro}
In the last decades, natural convex cones in the N\'eron-Severi space (and its dual) of algebraic  varieties are playing an important role \cite{Laz1}. Ample, nef, big and semiample cones are some of these cones. In particular, the cone of curves and its clousure in the usual topo\-logy (the Mori cone) are key objects in the minimal model program (\cite{Mori,Kaw,BirCasHac2010I,HacMck2010II,Casc2021}). Within the study of algebraic varieties another interesting tool is the so-called Cox ring. This ring encodes the geometry and the line bundles on a variety \cite{ArzDerHauLaf}. When $X$ is a smooth variety over an algebraically closed field $k$ and a finite set $\{D_i\}_{i=1}^m$ freely generates the Picard group of $X$, its Cox ring (introduced in \cite{Cox}) is defined as
$$
\operatorname{Cox}(X)=\bigoplus_{s_i\in\mathbb{Z}, 1\leq i\leq m} H^0 \left(X,\mathcal{O}_X\left(\sum_{i=1}^m s_iD_i\right)\right).
$$
A Mori dream space is a variety whose Cox ring is finitely generated as $k$-algebra. The behaviour of these spaces is optimum in the minimal model program \cite{HuKe}. Cox rings are rather studied in the last years \cite{LafVel,HauSub,Bro,HauKeiLaf,ArzBraHauWro}. Although the minimal model program essentially considers higher dimensional varieties, the above mentioned tools, cone of curves and Cox ring of smooth (projective) surfaces, are not well understood. In this paper, we focus on surfaces. To classify surfaces $X$ whose cone of curves $\operatorname{NE}(X)$ is (finite)  polyhedral is an open problem and the same happens with the property of having finitely generated Cox ring. Notice that $\operatorname{Cox}(X)$ finitely generated implies the polyhedrality of $\operatorname{NE}(X)$ and these two properties are equivalent for $K3$ surfaces \cite{ArtHauLaf}. This last result and some others on surfaces use the characterization of finite generation of Cox rings and some applications stated in Theorem 1 and Corollaries 1 and 2 of \cite{GalMonCox}.

We are mainly interested in rational surfaces $X$. By \cite{GalMonCox} it is known that $X$ is a Mori dream space when $K_X^2>0$ or $K_X^2=0$ and $-K_X$ not nef, $K_X$ being a canonical divisor of $X$. These results are included in the more general case, proved in \cite{TesVarVel}, which states that any rational surface such that $-K_X$ is big is a Mori dream space. Rational surfaces can be obtained by blowing up at a configuration of infinitely near points over the projective plane $\mathbb{P}^2$ or a Hirzebruch surface $\mathbb{F}_\delta$. The recent literature contains some families of Mori dream spaces obtained by careful choices of configurations over Hirzebruch surfaces \cite{TesVarVel,mus,FriasLahy2,RosaFrisMus2023} or of points in general position in these surfaces \cite{FriasLahy1}. Something similar happens to the polyhedrality of the cone of curves. Apart from the cases when the Cox ring is finitely generated, some numerical conditions (implying the polyhedrality of the cone of curves) are given in \cite{GalMon1,GalMon2,GalMonCox,GalMon}. The cone of curves of a rational surface is not always polyhedral \cite[Example 4.3]{CamGonz} and particular interest has the case when one blows up more than 9 very general points in the projective plane. In this case, the geometry of the nef and Mori cones is related to the Nagata conjecture \cite{Nag} (see \cite[Conjecture 2.2]{RoeSup}, and  \cite{GalMonMoy,GalMonMorMoy} for a valuative statement of the conjecture).

In this paper we prove that, if $X$ is a rational surface obtained by blowing up at a configutation of infinitely near points over a Hirzebruch surface $\mathbb{F}_\delta$, then there exist two positive integers $b\geq a$  such that, when $\delta\geq b$, $X$ is a Mori dream space and, when $\delta \geq a,$ the cone of curves $\operatorname{NE}(X)$ is polyhedral and minimally generated. To prove that $X$ is a Mori dream space, we use the before mentioned fact that it holds whenever $-K_X$ is big; no information is given about the structure of the Cox ring of $X$.

Now we state the specific result but, previously, we need some notation. Assume that $\pi:X\to\mathbb{F}_\delta$ is the composition of point blowups at a configuration of infinitely near points over $\mathbb{F}_\delta$ giving rise to the surface $X$. The topology of the exceptional curves on $X$ can be represented by the proximity graph (or equivalently the dual graph) of $\pi$. The arrowed proximity graph of $X$, APG$(X),$ is the proximity graph together with some arrows that show those exceptional divisors to which the strict transforms of the fibers (going through blown up points in $\mathbb{F}_\delta$) and the special section are transversal (see page  \pageref{apgg}). Then our main result is:

\begin{introtheorem}[Theorem \ref{Thm:Cone_equiv_conds} in the paper]\label{thm_intro_1}
Let $X$ be a rational surface obtained by a composition of point blowups $\pi:X\to\mathbb{F}_\delta$ over a Hirzebruch surface $\mathbb{F}_\delta$. There exist two positive integers $a(\operatorname{APG}(X,\pi))\leq b(\operatorname{APG}(X,\pi))$, depending only on the arrowed proximity graph $\operatorname{APG}(X,\pi)$, such that:
\begin{itemize}
\item[1)] If $\delta\geq a(\operatorname{APG}(X,\pi))$, then:
\begin{itemize}
\item[(a)] The cone of curves $\operatorname{NE}(X)$ is finite polyhedral and its extremal rays are given by the classes of the strict transforms on $X$ of the fibers on $\mathbb{F}_\delta$ going through the blown up points in $\mathbb{F}_\delta$, the special section, and the exceptional divisors of $\pi$. In this case we say that $\operatorname{NE}(X)$ is minimally generated.
\item[(b)] The nef cone $\operatorname{Nef}(X)$ is generated by the classes of the specific divisors shown in Lemma \ref{lema:Dual_cone_S}.
\end{itemize}
\item[2)] If $\delta\geq b(\operatorname{APG}(X,\pi))$, then $X$ is a Mori dream space.
\end{itemize}
\end{introtheorem}

For each specific surface given by $\pi$, the values $a(\operatorname{APG}(X,\pi))$ and $ b(\operatorname{APG}(X,\pi))$ introduced in Theorem \ref{thm_intro_1} can be obtained in an algorithmic way, see Remarks \ref{GetA} and \ref{GetBp}. However, explicit formulae can be given when the confi\-guration $\mathcal{C}$ of infinitely near points given by $\pi$ contains either only free points or it is a constellation (that is, it comes from blowing up at a unique proper point $p$ in $\mathbb{F}_\delta$ and finitely many infinitely near points), see  Corollaries \ref{cor:delta0} and \ref{cor:value_b}. These corollaries give formulae for $a(\operatorname{APG}(X,\pi))$ and a value $b'(\operatorname{APG}(X,\pi))$ which determines the value $b(\operatorname{APG}(X,\pi)):=\max\{a(\operatorname{APG}(X,\pi)),$ $b'(\operatorname{APG}(X,\pi))\}$. Note that Remark \ref{GetBp} actually determines $b'(\operatorname{APG}(X,\pi))$.

A configuration can be expressed as a union of chains of infinitely near points corresponding to divisorial valuations $\nu_j,1\leq j\leq m$. Denote by $\mathcal{C}_{\nu_j}=\{p_{j\, k}\}_{k=1}^{n_j}$ the chain of $\nu_j$. Write $\varphi_{j\,n_j}$ (respectively, $\varphi_{F_j},\varphi_{M_0}$) an analytically irreducible germ at $p_{j\,1}$ such that its strict transform is transversal to the exceptional divisor $E_{p_{j\,n_j}}$ at a general point of the exceptional locus (respectively, the germ at $p_{j\,1}$ of the fiber $F_j$ of $\mathbb{F}_\delta$ going through $p_{j\,1}$, the germ at $p_{j\,1}$ of the special section $M_0$ of $\mathbb{F}_\delta$). Consider the following intersection numbers:  $a_{j\,n_j}:=(\varphi_{M_0},\varphi_{j\,n_j})_{p_{j\,1}}$ and $b_{j\,n_j}:=(\varphi_{F_j},\varphi_{j\,n_j})_{p_{j\,1}}$. Then, the following result groups the main statements of the above corollaries.
\begin{introcorollary}
Let $X$ be a surface as in Theorem \ref{thm_intro_1} and $\mathcal{C}$ the configuration of infinitely near points defining $\pi.$
\begin{itemize}
\item[(i)] Assume that $\mathcal{C}$ has only free points. Suppose also that, for any divisorial valuation $\nu_j,1\leq j \leq m$, $\nu_j(\varphi_{M_0})=0$ and $\nu_j(\varphi_{F_j})=1$. Then
$$
    a(\operatorname{APG}(X,\pi))=\max\Bigg\lbrace\sum_{h=1}^s \# (\mathcal{C}_{\nu_{j_h}}) \ \bigg| \  {\footnotesize \begin{array}{l}
       \text{$\big\lbrace \mathcal{C}_{\nu_{j_h}} \big\rbrace_{h=1}^s$ is a set of pairwise}\\ \text{disjoint chains whose origins} \\ \text{belong to different fibers}\\
    \end{array}}\Bigg\rbrace,
    $$
  where $\# (\mathcal{C}_{\nu_{j_h}})$ stands for the cardinality of the chain $\mathcal{C}_{\nu_{j_h}}$, and

    $$b'(\operatorname{APG}(X,\pi))=\max\Bigg\lbrace \Bigg\lceil\sum_{h=1}^s \# (\mathcal{C}_{\nu_{j_h}})-2\Bigg\rceil^+ \ \bigg| \  {\footnotesize \begin{array}{l}
       \text{$\big\lbrace \mathcal{C}_{\nu_{j_h}} \big\rbrace_{h=1}^s$ is a set of pairwise}\\ \text{disjoint chains whose origins} \\ \text{belong to different fibers}\\
    \end{array}}\Bigg\rbrace,$$
    where  $\lceil x \rceil^+$ is defined as the ceil of a rational number $x$ if $x\geq 0$, and $0$ otherwise.

\item[(ii)] Assume that $\mathcal{C}$ is a constellation. Then
$$
a(\operatorname{APG}(X,\pi))=\max\Bigg\lbrace \Bigg\lceil\dfrac{\sum_{\lambda=1}^{n_{j}}\mult_{p_{j\,\lambda}}(\varphi_{j\,n_{j}})^2-2a_{j\,n_j}\,b_{j\,n_j}}{b_{j\,n_j}^2}\Bigg\rceil^* \ \bigg| \  1\leq j\leq m
\Bigg\rbrace \text{ and}
$$
$$b'(\operatorname{APG}(X,\pi))=\max\Bigg\lbrace \Bigg\lceil\dfrac{\sum_{\lambda=1}^{n_{j}}\mult_{p_{j\,\lambda}}(\varphi_{j\,n_{j}})-2(a_{j\,n_j}+b_{j\,n_j})}{b_{j\,n_j}}\Bigg\rceil^+ \ \bigg| \  1\leq j\leq m
\Bigg\rbrace,
$$
where $\lceil x \rceil^*$ is defined as the minimum positive integer upper bound of $\{x\}$.
\end{itemize}
\end{introcorollary}

The bounded negativity conjecture states that every surface $X$ in characteristic $0$ has bounded negativity and it has attracted a lot of attention in the last years  \cite{bauer2013,bauer2015,Harb2,Rou,Pok-Roe,GalMonMorPer}. It means that there exists a nonnegative integer $\alpha(X)$ depending only on $X$ such that $C^2\geq - \alpha(X)$ for any integral curve $C$ on $X$. Theorem \ref{thm_intro_1} proves, in particular, that most rational surfaces whose relatively minimal model is a Hirzebruch surface have bounded negativity. It is a consequence of the fact that those surfaces whose cone of curves is finite polyhedral have bounded negativity. Let us state the specific result.
\begin{introproposition}
\label{pr_intro}
Let $X$ be a rational surface obtained by a composition of point blowups $\pi:X\to\mathbb{F}_\delta$. Assume that $\delta \geq a(\operatorname{APG}(X,\pi))$, where $a(\operatorname{APG}(X,\pi))$ is the value introduced in Theorem \ref{thm_intro_1}. Then, $X$ has bounded negativity.
\end{introproposition}
Proposition \ref{pr_intro} shows that each proximity graph $\operatorname{PG}$ determines a positive integer $a(\operatorname{PG})$ such that the rational surfaces $X$ over a Hirzebruch surface $\mathbb{F}_\delta$ with proximity graph $\operatorname{PG}$ have bounded negativity whenever $\delta \geq a(\operatorname{PG})$. The integer $a(\operatorname{PG})$ is the maximum of the set of positive integers $$\mathfrak{A} := \{a(\operatorname{APG}(X,\pi)) \; | \; \pi \mbox{ is as in Proposition \ref{pr_intro} and has proximity graph } \operatorname{PG} \}$$
and it is of purely combinatorial nature in the sense that it depends only on the relative position among the exceptional divisors (and not on a specific choice of the blown up infinitely near points). Note that, according the forthcoming Corollary \ref{cor:delta0}, $\mathfrak{A}$ is a finite set.

This paper is structured as follows. Section \ref{sec:configuration} recalls the concepts of configuration, constellation and chain of infinitely near points over a surface. These concepts together with that of divisorial valuation of the quotient field of a surface, also presented in Section \ref{sec:configuration}, will be important in the paper. Section \ref{Sect3} introduces the arrowed proximity graph and our main results, Theorem \ref{Thm:Cone_equiv_conds} and Corollaries \ref{cor:delta0} and \ref{cor:value_b}, are stated and proved in Section \ref{sec:main_section}. This section has four subsections and contains two examples. The first one, divided in two parts, Examples \ref{New} and \ref{New2}, is simpler than the second one, which is successively developed in Examples \ref{ex:grafo_proximidad}, \ref{example_1} and \ref{example_2}. These examples are intended to help understand our results and the complexity of the involved tools. Finally, Example \ref{ex:a_b_different} proves that the values $a(\operatorname{APG}(X,\pi))$ and $ b(\operatorname{APG}(X,\pi))$ forcing the minimal generation of the cone of curves and the finite generation of the Cox ring can be different.

\section{Configurations, constellations and chains of infinitely near points}\label{sec:configuration}

In this paper we are interested in surfaces obtained by blowing up at finitely many proper or infinitely near points of a smooth surface. The blowup centers constitute a configuration (of infinitely near points) and enjoy some properties that will be useful. We devote this section to recall the main concepts and some properties related to configurations, being \cite{Cas,CamGonMon} our main references.

Let $X$ be a smooth projective surface over an algebraically closed field $k$.  Denote by Bl$_p(X)$ the smooth surface given by blowing up at a closed point $p\in X$ and by $E_p$ the created exceptional divisor. A \emph{constellation of infinitely near points over $X$} (\emph{constellation over $X$} for short) is a nonempty set of closed points $\mathcal{C}=\{p_l\}_{l=1}^N$ such that $p_1\in X_0:=X$, $p_{l+1}\in X_l:=\text{Bl}_{p_{l}}(X_{l-1})\xrightarrow{\pi_{l}}X_{l-1}$ for $1\leq l\leq N$ and $p_1$ is the image of $p_l,$ for all $l$ such that $2\leq l\leq N$, under the composition of blowups giving rise to $p_l$. The point $p_1$ is called the \emph{origin} of the constellation and $\pi_\mathcal{C}:=\pi_1\circ\cdots\circ \pi_N:X_\mathcal{C}:=X_{N}\to X_0$ denotes the composition of the sequence of blowups at the points of $\mathcal{C}$.

Given $l,l'\in\{1,\ldots,N\},$ a point $p_{l'}$ is infinitely near $p_l$ if $p_{l'}=p_l$ or $p_l$ is the image of $p_{l'}$ under the composition giving rise to $p_{l'}$. It is denoted by $p_{l'}\geq p_l$ and defines a partial ordering on $\mathcal{C}.$ Clearly, $p_l\geq p_1$ for all $l\in\{1,\ldots,N\}$. We say that a point $p_l\in\mathcal{C}$ is \emph{maximal} if it is maximal with respect to $\geq$. When $\geq$ is a total order, $\mathcal{C}$ is called a \emph{chain}. For any constellation $\mathcal{C}$ and any point $p\in \mathcal{C}$, we can define the chain $\mathcal{C}^p:=\{q\in \mathcal{C}\,| \, p\geq q\}$. The number of points in $\mathcal{C}^p$ different from $p$ is called \emph{level} of $p$, denoted $\ell(p)$. The  origin of a constellation is the unique point of level $0$ of the constellation.

A \emph{configuration $\mathcal{C}$ of infinitely near points over $X$} (\emph{configuration over $X$} for short) is a finite disjoint union of constellations whose origins are points in $X$. Clearly, $\mathcal{C}$ defines a surface $X_\mathcal{C}$. Following \cite{CamGonMon}, fixed two configurations, $\mathcal{C}$ and $\mathcal{C}'$ over $X$, we can identify them if there exist an automorphism $\sigma:X\to X$ and an isomorphism $\sigma':X_\mathcal{C}\to X_{\mathcal{C}'}$ such that $\pi_{\mathcal{C}'}\circ\sigma'=\sigma\circ\pi_\mathcal{C}.$

Let $\mathcal{C}=\{p_l\}_{l=1}^N$ be a configuration of infinitely near points over $X$. We say that $p_{l'}$ is \emph{proximate} to $p_l$, $1\leq l,{l'}\leq N,$ denoted $p_{l'}\to p_l$, if $p_{l'}$ belongs to the strict transform of the exceptional divisor $E_{p_l}$ on $X_{l'-1}.$  According to \cite[Lemma 1.1.16]{Alb}, the exceptional divisor of a composition of blowups centered at closed points is the support of a simple normal crossing divisor \cite[Definition 4.1.1]{Laz1}. A point of $\mathcal{C}$ is \emph{satellite} if  it is proximate to other two points in $\mathcal{C}$. Otherwise, it is \emph{free}.

 \subsection{Divisorial plane valuations}\label{subsec:div_val}

Let $K$ be a field and set $K^*=K\setminus \{0\}$. A \emph{valuation} $\nu$ of $K$ is a surjective map $\nu :K^*\to G,$ where $G$ is a totally ordered commutative group, such that
$$
\nu(f+g)\geq \min\{\nu(f),\nu(g)\} \text{ and } \nu(fg)=\nu(f)+\nu(g),
$$
$\text{for }f,g\in K^*.$ The group $G$ is called the \emph{value group} of $\nu$. In addition, we denote by $R_\nu$ the local ring $R_\nu:=\{f\in K\setminus \{0\} \ | \ \nu(f)\geq 0\}\cup \{0\},$ called the \emph{valuation ring} of $\nu$, and by $\mathfrak{m}_\nu$ its maximal ideal $\mathfrak{m}_\nu:=\{f\in K\setminus\{0\} \ | \ \nu(f)>0\}\cup \{0\}$. If $K$ is the quotient field of a two-dimensional regular local ring $(R,\mathfrak{m})$ whose maximal ideal is $\mathfrak{m}$ and $\nu$ is {\it centered} at $R$, i.e. $R\cap\mathfrak{m}_\nu=\mathfrak{m}$, $\nu$ is said to be a \emph{plane valuation}. By \cite{Spiv} (see \cite{ZarSam}) plane valuations are in one-to-one correspondence with (not necessarily finite) sequences of blowups whose configurations over Spec$\,R$  are chains and   the origin is the closed point $p\in\text{Spec}\, R$ (associated to $\mathfrak{m}$). Divisorial (plane) valuations are those corresponding to finite sequences.
\medskip

For us, a divisorial valuation $\nu$ of (a smooth projective surface) $X$ is a valuation $\nu$ of the quotient field of the ring $R:=\mathcal{O}_{X,p},$ where $p$ is a closed point in $X,$ centered at $R$. The valuation $\nu$ defines a chain giving rise to a sequence of blowups
\begin{equation}\label{Eq_sequence_div_valuation}
\pi: X_\nu:=X_{n}\xrightarrow{\pi_{n}} X_{n-1}\rightarrow \cdots \rightarrow X_1 \xrightarrow{\pi_1} X_0=X,
\end{equation}
where $\pi_1$ is the blowup at the closed point $p=p_1\in X$ defined by the maximal ideal $\mathfrak{m}_1=\mathfrak{m}$ of $R$ and $\pi_{l}, l\geq 2,$ is the blowup at the unique closed point $p_{l}\in X_{l-1}$ belonging to the exceptional divisor $E_{p_{l-1}}$ created by $\pi_{l-1}$ and such that the plane valuation $\nu$ is centered at $\mathcal{O}_{X_{l-1},p_{l}}$. This type of sequences, where each blowup is performed at a point of the last created exceptional divisor, are also named \emph{simple}. Notice that the valuation $\nu$ satisfies $\nu(f)= \ord_{E_{p_n}}(f)$ for all $f\in K^*$ \cite{Spiv}.

Let $\mathcal{C}_\nu=\{p_l\}_{l=1}^n$ be the chain defined by a divisorial valuation $\nu$. For $l\geq 1$, $\mathfrak{m}_l$ denotes the maximal ideal corresponding to the closed point $p_l$. The value $\nu(\mathfrak{m}_l)$ is defined as $\nu(\mathfrak{m}_l):=\min\{\nu(f)\ | \ f\in\mathfrak{m}_l\setminus\{0\}\}$. The sequence  $(\nu(\mathfrak{m}_l))_{l=1}^n$ is called the \emph{sequence of values} of $\nu$. Values in this sequence satisfy the so-called \emph{proximity equalities} \cite[Theorem 8.1.7]{Cas}:
\begin{equation}\label{proximity_equalities}
\nu(\mathfrak{m}_l)=\sum_{p_{l'}\to p_l} \nu(\mathfrak{m}_{l'}), \, l\geq 0.
\end{equation}
When considering a divisorial valuation of $X$, we denote by $\varphi_{C}$ the germ of a curve $C$ on $X$ at the point $p$. Additio\-nally, $\varphi_l$ denotes an analytically irreducible germ at $p$ such that its strict transform on $X_l$ is transversal to the exceptional divisor $E_{p_l}$ at a general point of the exceptional locus. The value $\text{mult}_{p_{l'}}(\varphi_l)$ (respectively, $\text{mult}_{p_{l'}}(\varphi_C)$), $1\leq l,{l'} \leq n,$ denotes the multiplicity of the strict transform of $\varphi_l$ (respectively, $\varphi_C$) at $p_{l'}$. It holds that $\nu(\mathfrak{m}_l)=\text{mult}_{p_l}(\varphi_{n})$. We use frequently, without any mention, the so-called \emph{Noether formula for valuations}, which we recall here for the convenience of the reader. A proof can be found in \cite[Theorem 8.1.6]{Cas}.

\begin{proposition}
Let $\nu$ be a divisorial valuation of $X$, with associated configuration of infinitely near points  $\mathcal{C}_\nu:=\{p_l\}_{l=1}^n$, and let $C$ be a curve on $X$. Then
$$\nu(\varphi_C)=(\varphi_n,\varphi_C)_p=\sum_{l=1}^n {\rm mult}_{p_l}(\varphi_n)\cdot {\rm mult}_{p_l}(\varphi_C).$$
\end{proposition}

The sequence of maximal contact values $(\overline{\beta}_w)_{w=0}^{g+1}$ of a divisorial valuation $\nu$ is a well-known invariant related to divisorial valuations. See \cite[(1.5.3)]{DelGalNun} for the definition. The sequence $(\overline{\beta}_w)_{w=0}^{g}$ is the minimal generating set of the {\it semigroup of values} of $\nu$, $S(\nu) :=\nu(R\setminus\{0\})$, \cite[Remark 6.1]{Spiv}, and the value $\overline{\beta}_{g+1}$ coincides with the inverse of the volume of $\nu$, $[\text{vol}(\nu)]^{-1},$ where
$$
\text{vol}(\nu):=\lim_{\alpha\to\infty}\dfrac{\dim_k(R/ \mathcal{P}_\alpha)}{\alpha^2/2}
$$
and $\mathcal{P}_\alpha=\{f\in R \ | \ \nu(f)\geq\alpha\}\cup \{0\}$ (see \cite[Remark 2.3]{GalMonMoyNic2}).

\section{The arrowed proximity graph}\label{Sect3}

Any rational smooth projective surface can be obtained by blowing up at a configuration of infinitely near points over the projective plane or a Hirzebruch surface $\mathbb{F}_\delta$, $\delta\geq 0$. We are mainly interested in the case where the configuration is over $\mathbb{F}_\delta$. This section is devoted to recall some known facts on the obtained rational surfaces and to introduce a new concept related to them: the arrowed proximity graph. It will be useful in this article.

Recall that $\mathbb{F}_\delta$ is a ruled surface (over an algebraically closed field $k$) together with a projective morphism $\psi:\mathbb{F}_\delta \to \mathbb{P}^1,$ where $\mathbb{P}^1$ is the projective line over $k$. Denote by $F$ a fiber of the projective morphism $\psi$.  Following \cite{Har}, there exists a section $M_0$ of $\psi$ with self-intersecion $-\delta$. When $\delta$ is positive, $M_0$ is the unique reduced and irreducible curve with negative self-intersection and we call it \emph{special section}. In the case $\delta =0$, $M_0$ is simply a section of self-intersection zero. In addition, denote by $M$ a section which is linearly equivalent to the divisor $M_0+\delta F$ and such that $M_0\cap M=\emptyset$. The Picard group Pic$(\mathbb{F}_\delta)$ of $\mathbb{F}_\delta$ is generated by the classes of $F$ and  $M$ and it holds that $F^2=0,F\cdot M=1$ and $M^2=\delta$.

As we said above, we center our interest in configurations over Hirzebruch surfaces.
That is, we consider any rational surface $X$ given by blowing up at a configuration $\mathcal{C}=\{p_l\}_{l=1}^N$ over a Hirzebruch surface $\mathbb{F}_\delta$.
Set $\pi = \pi_{\mathcal{C}}:X\to \mathbb{F}_\delta$ the composition of blowups  giving rise to $X$ and defined by $\mathcal{C}.$ It is clear that $\mathcal{C}$ is a disjoint union of constellations $\mathcal{C}_i, 1\leq i\leq r$. Denote by $F_i$ the fiber on $\mathbb{F}_\delta$ passing through the origin of $\mathcal{C}_i$, for $i=1,\ldots,r$ (notice that the fibers $F_1,\ldots,F_r$ need not to be different).

Denote by $E_l:=E_{p_l}$ the exceptional divisor created after blowing up at the point $p_l$ and, abusing the notation, set $E_l$ (respectively, $E^*_l$) the strict (respectively, total) transform of the exceptional divisor $E_l$ on $X$. Also, for each divisor $D$ on $\mathbb{F}_\delta$, $\tilde{D}$ (respectively, $D^*$) denotes the strict (respectively, total) transform of the divisor $D$ on $X$.

Let $\{p_{t_j}\}_{j=1}^m$ be the set of maximal points of $\mathcal C$ with respect to the ordering $\geq $. Each one of the divisors $E_{t_j}$ (which we name \emph{final exceptional divisors}) defines a divisorial valuation $\nu_j:=\nu_{E_{t_j}}$. Set ${\mathcal C}_{\nu_j}=\{p_{j\,k}\}_{k=1}^{n_j}\subset\mathcal{C}$ the chain defined by $\nu_j$; clearly $p_{j\,1}\in \mathbb{F}_{\delta}$ and $p_{j\,k}\to p_{j\,k-1}$ for all $k\in \{2,\ldots, n_j\}$. The configuration ${\mathcal C}$ can be expressed as follows:
\begin{equation}\label{val}
{\mathcal C}=\bigcup_{j=1}^m {\mathcal C}_{\nu_j}.
\end{equation}
In addition, we use the notation $E_{j\,k}$ to indicate the exceptional divisor created by blowing up at $p_{j\,k}$ and $\varphi_{j\,k}$ stands for an analytically irreducible germ at $p_{j\,1}$ such that its strict transform is transversal to the exceptional divisor $E_{j\,k}$ at a general point. Notice that the origin of each chain ${\mathcal C}_{\nu_j}$ belongs to a unique fiber on $\mathbb{F}_\delta$, and each one of these fibers contains the origin of at least one divisorial valuation $\nu_j$. Sometimes we use the notations $F_{\nu_j}$ or $F_j$ to indicate the fiber going through the origin of the chain given by $\nu_j$. In order to avoid excessive cumbersome notation, a subscript $i$ (respectively, $j$) in $F_i$ (respectively, $F_j$) indicates that $F_i$ (respectively, $F_j$) is the fiber going through the origin of the constellation $\mathcal{C}_i$, $1\leq i\leq r$ (respectively, chain $\mathcal{C}_{\nu_j}$, $1\leq j\leq m$).

Next we introduce an important object in this paper. It is attached to a rational surface $X$ obtained by a composition of blowups $\pi:X\rightarrow \mathbb{F}_\delta$. It is named the \emph{arrowed proximity graph of} $(X,\pi)$ and it is denoted by $\operatorname{APG}(X,\pi)$. Denoting by $\mathcal{C}$ the configuration of centers of $\pi$ and keeping the above notation, the graph $\operatorname{APG}(X,\pi)$\label{apgg} is formed by the disjoint union of $r$ labeled rooted trees with arrows, each one of them corresponding to a constellation $\mathcal{C}_i$. The vertices of each tree correspond to the points in $\mathcal{C}_i$ and are labeled with $p_{j\,k}$. Notice that a point in $\mathcal{C}_i$ can have different labels if it corresponds to different valuations. Without loss of generality, we use any of them. Two points $p,q\in\mathcal{C}_i$ are joined by an edge if $q>p$ and $\ell(q)=\ell(p)+1.$ The origin of $\mathcal{C}_i$ determines the root. We also add edges joining those points $p$ and $q$ in $\mathcal{C}_i$ such that $q\to p$ and $\ell(q)>\ell(p)+1,$ although for simplicity when depicting the graph we delete those edges one can deduce from others. Indeed, if $r, p, q \in \mathcal{C}$, $r \geq q \geq p$, $q \neq p$ and $r \rightarrow p$, then $r \rightarrow q$ by \cite[Theorem 1.6 (iii)]{CamGozJal} and the edge joining $q$ and $p$ can be deleted because it is deduced from that joining $r$ and $p$. We complete our graph by adding arrows with a label $\tilde{F}_i$ or $\tilde{M}_0$ to those vertices of the graph corresponding to the last points of the constellations $\mathcal{C}_i$ through which the strict transform of the fibers $F_i$ or the special section $M_0$ go. Notice that $\operatorname{APG}(X,\pi)$ depends only on the configuration $\mathcal{C}$, the fibers $F_i$ and the special section $M_0$.

If one deletes the arrows in the arrowed proximity graph $\operatorname{APG}(X,\pi)$, one gets the {\it proximity graph} of the pair $(X,\pi)$. A less explicit version of this last graph was introduced in \cite[Section 1]{CamGozJal} for configurations over smooth varieties of dimension larger than one. In our context, the proximity graph determines and it is determined by other combinatorial objects like the dual graph or the Enriques diagram of the configuration $\mathcal{C}$ (see, for example, \cite{Spiv, Galog, Cas} for these last concepts). We use the proximity graph because it eases treatment with multiplicities and proximity equalities. The concepts introduced in this section and the preceding one are related with the notion of {\it bubble space} \cite[Section 7.3.2]{Dolgachev}.

\begin{figure}[htbp]
    \includegraphics[scale=1.5]{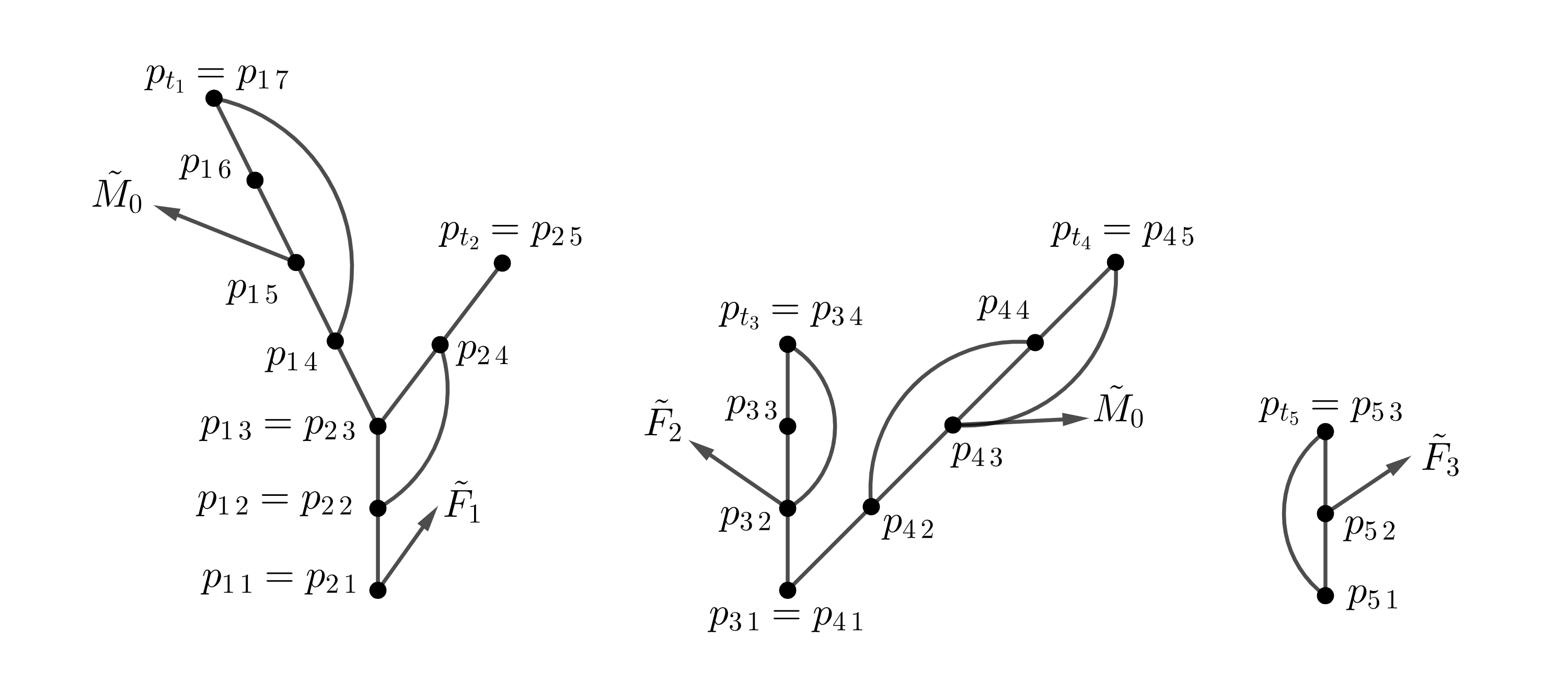}
    \caption{Arrowed proximity graph $\operatorname{APG}(X,\pi)$ of the rational surface $X$ in Examples \ref{ex:grafo_proximidad}, \ref{example_1} and \ref{example_2}.}\label{fig_conf_F_delta}
\end{figure}

\begin{example}\label{ex:grafo_proximidad}
Figure \ref{fig_conf_F_delta} shows the arrowed proximity graph $\operatorname{APG}(X,\pi)$ of a rational surface $X$ obtained by blowing up the surface $\mathbb{F}_\delta$ at a configuration $\mathcal{C}=\{p_l\}_{l=1}^{20}=\cup_{i=1}^3\mathcal{C}_i$ of twenty infinitely near points.
Notice that the graph $\operatorname{APG}(X,\pi)$ shows the topology of the exceptional locus of the composition $\pi:X\to\mathbb{F}_\delta$ (when the ground field is $\mathbb{C}$) and also indicates the irreducible components of the exceptional locus where the strict transforms of the fibers $F_i$ and the special section $M_0$ are transversal. Moreover we indicate the maximal points with the labels $p_{t_j},1\leq j\leq 5$. Note that the edge joining $p_{16}$ and $p_{14}$ has been deleted because the fact $p_{16} \rightarrow p_{14}$ can be deduced from the proximity $p_{17} \rightarrow p_{14}$.
\end{example}

\section{The effective cone and the Cox ring}\label{sec:main_section}

Along this section, unless otherwise stated, $X$ denotes a rational surface obtained by a composition of blowups $\pi:X\to\mathbb{F}_\delta$. We also keep the notation introduced in the preceding section.

\subsection{The cone $\mathfrak{C}(X)$ and its dual} \label{41}

Denote by $\NS(X)$ the Néron-Severi group of the surface $X$ and set $\NS_\mathbb{R}(X):=\NS(X)\otimes \mathbb{R}$. As usual, the symbol $\cdot$ stands for the intersection product. Let $\NE(X)$ (respectively, $\overline{\NE}(X)$), $\Nef(X)$) be the cone of curves (respectively, the \emph{pseudoeffective cone} or \emph{Mori cone}, the \emph{nef cone}) of $X$.

Consider the convex cone $\mathfrak{C}(X)$ (embedded into $\NS_\mathbb{R}(X)$) generated by the following set of classes of  divisors
\begin{equation}\label{generators}
\big\{[\tilde{F}_i]\big\}_{i=1}^r\cup \big\{[\tilde{M}_0]\big\}\cup\big\{[E_l]\big\}_{l=1}^N.
\end{equation}
The dual cone of $\mathfrak{C}(X)$ is defined as
$$
\mathfrak{C}^\vee(X):=\big\{[D]\in\text{NS}_{\mathbb{R}}(X) \ | \ [D]\cdot [C]\geq 0, \text{ for all }[C]\in \mathfrak{C}(X)\big\}.
$$
The following result determines a set of generators of the cone $\mathfrak{C}^\vee(X)$ and it will be crucial for our forthcoming main result.

\begin{lemma}\label{lema:Dual_cone_S}
Let $X$ be a rational surface obtained by blowing up at a configuration $\mathcal{C}$ of infinitely near points over a Hirzebruch surface $\mathbb{F}_\delta$, with $\delta>0$. Keep the notation introduced in this section. The cone $\mathfrak{C}^\vee(X)$ is generated by the classes in $\NS_\mathbb{R}(X)$ of the following (not necessarily distinct) divisors:
\begin{itemize}
\item[•] $F^*$ and $M^*.$
\item[•] For each $j,1\leq j\leq m,$ corresponding to the divisorial valuations defined by the final exceptional divisors,
$$
\Lambda_{j\,k}:=\Lambda_k(\nu_{j}):= a_{j\,k} F^* +  b_{j\,k}M^* - \sum_{\lambda=1}^k\mult_{p_{j\,\lambda}}(\varphi_{j\,k})E_{j\,\lambda}^*,
$$
where $k$ runs over the set $\{k \ | \ 1\leq k \leq n_j\}$ and $a_{j\,k}$ (respectively, $b_{j\,k}$) denotes the intersection multiplicity $(\varphi_{M_0},\varphi_{j\,k})_{p_{j\,1}}$ (respectively, $(\varphi_{F_{j}},\varphi_{j\,k})_{p_{j\,1}}$), $F_{j}:=F_{\nu_j}$ being the fiber associated to the divisorial valuation $\nu_j.$
\item[•] For each subset $\{j_1,\ldots,j_s\}\subseteq \{1,\ldots,m\}$ of cardinality $s$, $1<s\leq r$, such that the chains in the set $\big\lbrace \mathcal{C}_{\nu_{j_h}} \big\rbrace_{h=1}^s$ are pairwise disjoint and their origins  belong to different fibers on $\mathbb{F}_\delta$, 
\begin{align*}
W_{(j_1\,k_1),\ldots,(j_s\,k_s)}:=&
\left(\sum_{h=1}^s z_{j_{h}\,k_{h}}\cdot a_{j_{h}\,k_{h}}\right)F^*+
\left(z_{j_1\,k_1}\cdot b_{j_{1}\,k_{1}}\right)M^*
\\ &
-\sum_{h=1}^s \sum_{\lambda=1}^{k_h} z_{j_{h}\,k_{h}}\mult_{p_{j_h\,\lambda}}(\varphi_{j_{h}\,k_{h}})E_{j_{h}\,\lambda}^*,
\end{align*}
where $ k_h,1\leq h\leq s,$ runs over the set $\{k_h \ | \ 1\leq k_h \leq n_{j_h}\}$ and $(z_{j_{1}\,k_{1}},\ldots,z_{j_{s}\,k_{s}})$ is any element in $(\mathbb{Z}_{>0})^s$ satisfying the following  equalities:
$$
z_{j_{1}\,k_{1}}\cdot b_{j_{1}\,k_{1}}=\cdots =z_{j_{s}\,k_{s}}\cdot b_{j_{s}\,k_{s}}.
$$
\end{itemize}
\end{lemma}

\begin{proof}
According to \cite[Section 1.2]{Ful2}, the cone $\mathfrak{C}(X)$ can be regarded as the intersection of the half-spaces defined as follows
$$
\mathcal{H}_\sigma:=\{z\in\NS_\mathbb{R}(X) \ | \ u_\sigma\cdot z\geq 0\},
$$
	where $\sigma$ ranges over the faces of $\mathfrak{C}(X)$ of codimension one  and $u_{\sigma}$ is an element in $ \mathfrak{C}^\vee(X)$ such that $\sigma=\mathfrak{C}(X)\cap u_\sigma^\perp$. As a consequence, the $u_\sigma$'s generate the convex cone $\mathfrak{C}^\vee(X)$. This implies that it suffices to check that the classes of the divisors described in the statement belong to $\mathfrak{C}^\vee(X)$, to consider every codimension one linear subspace $\mathcal{L}$ generated by $N+1$ elements in (\ref{generators}) and verify that its orthogonal complement $\mathcal{L}^\perp$ is generated by the class of some of the divisors given in the statement.

Let $\mathcal{L}$ be as above. Then, there exists a nonzero divisor
\begin{equation}\label{eq:prop_extremalrays_nefcone}
D:=aF^*+bM^*-\sum_{l=1}^N c_l E_l^*,
\end{equation}
such that $a, b$ and $c_l,1\leq l\leq N$, are nonnegative integers, and its class $[D]$ belongs to $\mathfrak{C}^\vee(X)$ and generates the orthogonal subspace  $\mathcal{L}^\perp$. Let us see the different possibilities for $\mathcal{L}$ and $D$.

If $\mathcal{L}$ contains all the classes of the strict transforms of the exceptional divisors, there are two possible cases:
\begin{itemize}
\item[$-$] $\mathcal{L}=\langle \{[\tilde{F}_i]\}\cup \{[E_l]\}_{l=1}^N \rangle$ for some $1\leq i\leq r$ and, therefore, $\mathcal{L}^\perp=\langle [F^*]\rangle$, and
\item[$-$] $\mathcal{L}=\langle \{[\tilde{M}_0]\}\cup \{[E_l]\}_{l=1}^N \rangle$ and, therefore, $\mathcal{L}^\perp=\langle [M^*] \rangle.$
\end{itemize}
Since $[F^*]$ and  $[M^*]\in \mathfrak{C}^{\vee}(X)$, $F^*$ and $M^*$ are two possibilities for the divisor $D$.

Now we study the case when $\mathcal{L}$ is generated by the class of $\tilde{M}_0$ and the class of a fiber $\tilde{F}_i,$ $1\leq i\leq r,$  and the remaining generators are classes of  strict transforms of exceptional divisors. That means that
\begin{equation}\label{eq:lema_nefdiv_hyperplane_lambda}
\mathcal{L}=\langle \{[\tilde{F}_i],[\tilde{M}_0] \} \cup \{[E_l]\}_{1\leq l\leq N, l\neq k'} \rangle,
\end{equation}
for some index $k'\in \{1,\ldots,N\}.$ With the previous notation, assume that $E_{k'}$ is created by blowing up at the point $p_{j\,k}$ for some $j\in\{1,\ldots,m\}$ and $k\in\{1,\ldots,n_j\}$. A divisor $D$ as in \eqref{eq:prop_extremalrays_nefcone} such that $[D]\in\mathfrak{C}^{\vee}(X)$ and it is ortho\-gonal to a linear space $\mathcal{L}$ as in \eqref{eq:lema_nefdiv_hyperplane_lambda} must satisfy:
\begin{equation}\label{eq:conds_lambda}
\begin{array}{r}
[D]\cdot [\tilde{F}_{i}]=0, [D]\cdot [\tilde{M}_0]=0, [D]\cdot [E_l]=0, 1\leq l\leq N, \text{ for } l\neq k', [D]\cdot [E_{k'}]> 0 \\ \text{and }  [D]\cdot [\tilde{F}_{\tau}]\geq 0, \text{ for } 1\leq \tau\leq r, \tau\neq i.
\end{array}
\end{equation}
Note that $[D]\cdot [E_{k'}]> 0$ is mandatory in order to obtain a nonvanishing divisor. For simplicity we can assume that $c_{k'}=[D]\cdot [E_{k'}^*]= [D]\cdot [E_{k'}]=\mult_{j\,k}(\varphi_{j\,k})=1$. In addition, it holds that
$$
0=[D]\cdot [E_l]=c_l-\sum_{p_\alpha\to p_l}c_\alpha, 1\leq l\leq N, \text{ for } l\neq k'.
$$
Therefore, the proximity equalities and the previous equation prove that the coefficients $c_l$ of $D$ are:
\begin{equation}\label{eq:lema_nefdiv_hyperplane_lambda_c_ell}
c_l= \Bigg\{
\begin{array}{cl}
\mult_{p_{j\,\lambda}}(\varphi_{j\,k}), & \text{when }E_l\text{ is created by blowing up at } p_{j\,\lambda}, 1\leq \lambda\leq k,\\[2mm]
0, & \text{otherwise.}
\end{array}
\end{equation}

The class $[\tilde{F}_j]$ corresponding to the fiber attached to the valuation $\nu_j$ must be a generator of $\mathcal{L}$ in \eqref{eq:lema_nefdiv_hyperplane_lambda}. Otherwise, reasoning by contradiction, if $[\tilde{F}_i]\neq [\tilde{F}_j]$ is a generator, the facts that $[D]\cdot [\tilde{F}_{i}]=0$ and that $F_i$ does not pass through $p_{j\,1}$ prove $b=0$ in the expression \eqref{eq:prop_extremalrays_nefcone} of the divisor $D$ we are looking for. Hence $[D]\cdot [\tilde{F}_{j}]< 0$ because $b=0$ and $F_j$ goes through $p_{j\,1}$, a contradiction.

It remains to  compute the coefficients $a$ and $b$ of $D$. From \eqref{eq:conds_lambda}, one gets
$$
\begin{array}{l}
0=[D]\cdot [\tilde{F}_{j}]=b-\sum_{\lambda=1}^{k}\mult_{p_{j\,\lambda}}(\varphi_{F_j})\mult_{p_{j\,\lambda}}(\varphi_{j\,k}) \text{ and }\\[2mm]
0=[D]\cdot [\tilde{M}_{0}]=a-\sum_{\lambda=1}^{k}\mult_{p_{j\,\lambda}}(\varphi_{M_0})\mult_{p_{j\,\lambda}}(\varphi_{j\,k}).
\end{array}
$$
By the Noether formula, we obtain that  $a=(\varphi_{M_0},\varphi_{j\,k})_{p_{j\,1}}$ and $b=(\varphi_{F_{j}},\varphi_{j\,k})_{p_{j\,1}}.$ This proves that the divisors $\Lambda_{j\,k}$ are generators of the cone $\mathfrak{C}^\vee(X)$.
\medskip

To conclude, we are going to show that we have to add the remaining divisors described in the statement to generate $\mathfrak{C}^\vee(X)$. The remaining linear subspaces $\mathcal{L}$ to be considered must be the following ones. Either
\begin{equation}\label{eq:prop_extremalrays_dualcone_1_part_1}
\mathcal{L}=\mathcal{L}_1^s:=\langle \{[\tilde{F}_{i_h}]\}_{h=1}^s\cup\{[\tilde{M}_0]\} \cup \{[E_l]\}_{1\leq l\leq N, l\neq k_{1}',\ldots,k_{s}'} \rangle,
\end{equation}
for some $s$ such that $1<s\leq r,$ where $1\leq i_h\leq r$ and $s$ different values $1\leq k_h'\leq N$, $1\leq h\leq s,$ are fixed, or
\begin{equation}\label{eq:prop_extremalrays_dualcone_1_part_2}
\mathcal{L}=\mathcal{L}_2^s:=\langle \{[\tilde{F}_{i_h}]\}_{h=1}^{s+1}\cup \{[E_l]\}_{1\leq l\leq N, l\neq k_{1}',\ldots,k_{s}'} \rangle,
\end{equation}
for some $s$ such that $1\leq s\leq r-1,$ where $1\leq i_h\leq r$ and $s$ different values $1\leq k_h'\leq N$, $1\leq h\leq s$, are fixed. Notice that we do not consider the value $s=1$ in the case $\mathcal{L}=\mathcal{L}_1^s$ because, then, $\mathcal{L}^\perp$ would be spanned by the class of one of the divisors $\Lambda_{jk}$.

As before, when considering $\mathcal{L}_1^s$, we look for a divisor $D$ as in \eqref{eq:prop_extremalrays_nefcone} such that
\begin{equation}\label{eq:lema_nefdiv_cond_L1}
\begin{array}{r}
[D]\cdot [\tilde{F}_{i_h}]=0, \text{ for } 1\leq h\leq s,[D]\cdot [\tilde{M}_0]=0, [D]\cdot [E_l]=0, 1\leq l\leq N, \text{ for } l\neq k_{1}',\ldots,k_{s}', \\[2mm]
[D]\cdot [E_{k_{h}'}]> 0, \text{ for } 1\leq h\leq s, \text{ and } [D]\cdot [\tilde{F}_{\tau}]\geq 0, \text{ for } \tau\in\{1,\ldots,r\} \setminus\{i_1,\ldots,i_s\}.
\end{array}
\end{equation}
Otherwise, $D$ must satisfy
\begin{equation*}
\begin{array}{c}
[D]\cdot [\tilde{F}_{i_h}]=0, \text{ for } 1\leq h\leq s+1, [D]\cdot [E_l]=0, 1\leq l\leq N, \text{ for } l\neq k_{1}',\ldots,k_{s}', \\[2mm]
[D]\cdot [E_{k_{h}'}]> 0, \text{ for } 1\leq h\leq s,[D]\cdot [\tilde{M}_0]\geq 0, \text{ and }\\[2mm]
 [D]\cdot [\tilde{F}_{\tau}]\geq 0, \text{ for } \tau\in\{1,\ldots,r\} \setminus\{i_1,\ldots,i_{s+1}\}.
\end{array}
\end{equation*}
In both cases we add the inequalities $[D]\cdot [E_{k_{h}'}]> 0, 1\leq h\leq s,$ in order to compute divisors different from those obtained with smaller values of $s$. Suppose that  $E_{k_1'},\ldots,E_{k_s'}$ are created by blowing up at the points $p_{j_1\,k_1},\ldots,p_{j_s\,k_s},$ respectively.
From the proximity equalities and the equalities
\begin{equation}\label{eq:lema_nefdiv_mult_E_ell}
0=[D]\cdot [E_l]=c_l-\sum_{p_\alpha\to p_l}c_\alpha, 1\leq l\leq N,  l\neq k_1',\ldots,k_s',
\end{equation}
one can deduce that the coefficients $c_l$ of $D$ are of the form
$$
c_l= \Bigg\{
\begin{array}{cl}
\omega_{j_h\,\lambda}, & \!\!\!\!\! \text{when }E_l\text{ is created after blowing up } p_{j_h\,\lambda}, 1\leq \lambda\leq k_h\text{ and }1\leq h\leq s.\\[2mm]
0, & \!\!\!\!\! \text{otherwise,}
\end{array}
$$
where $\omega_{j_h\,\lambda}$ is a positive integer.

To complete our proof, we have to take into account the following facts:	

1. \emph{All the fibers $F_{i_h}$ appearing in equations \eqref{eq:prop_extremalrays_dualcone_1_part_1} and \eqref{eq:prop_extremalrays_dualcone_1_part_2} must be of the form
$F_{j_h}=F_{\nu_{j_h}}$, where $1\leq h\leq s$}. Indeed, we are going to prove that the opposite case cannot hold. Firstly assume that no fiber is of the mentioned form. Then, $b=[D]\cdot [\tilde{F}_{i_h}]=0$ for $1\leq h\leq s$ since no fiber passes through $p_{j\,1},$ but $[D]\cdot [\tilde{F}_{j_h}]<0$ by the  proximity equalities, which is a contradiction. The other possibility is that some fiber (say $F_{j_h}$) is of the mentioned form and other one (say $F_{i_\ell}$) is not. Then $b>0$ from the condition $[D]\cdot [\tilde{F}_{j_h}]=0$ but also $b=[D]\cdot [\tilde{F}_{i_\ell}]=0$, which is again a contradiction. In particular, \emph{this fact discards to consider linear subspaces $\mathcal{L}_2^s$}.

2. \emph{The chains in the set $\big\lbrace \mathcal{C}_{\nu_{j_h}} \big\rbrace_{h=1}^s$ must be pairwise disjoint.} Otherwise, among the generators of $\mathcal{L}_1^s$ given in (\ref{eq:prop_extremalrays_dualcone_1_part_1}), it must appear the class of one of the fibers at least twice, and this contradicts the fact that $\mathcal{L}_1^s$ has codimension one.

3. \emph{The origins of the configurations $\mathcal{C}_{\nu_{j_h}}, 1\leq h\leq s,$ belong to different fibers on $\mathbb{F}_\delta$.} Other\-wise, at least two origins belong to the same fiber and then the linear subspace $\mathcal{L}_1^s$ does not have codimension one, a contradiction.

By the above Fact 2, the proximity equalities and \eqref{eq:lema_nefdiv_mult_E_ell} one
has that
$$
c_l= \Bigg\{
\begin{array}{cl}
z_{j_h\,k_h} \mult_{p_{j_h\,\lambda}}(\varphi_{j_h\,k_h}), & {\small \begin{array}{l}\text{when }E_l\text{ is created after blowing up } p_{j_h\,\lambda},\\ 1\leq \lambda\leq k_h\text{ and }1\leq h\leq s.\end{array}}\\[2mm]
0, & \ \, \small{\text{otherwise,}}
\end{array}
$$
for some positive integers $z_{j_h\,k_h}$, $1\leq h\leq s$.

It only remains to obtain the coefficients $a$ and $b$ of $[D]$ as in \eqref{eq:prop_extremalrays_nefcone}. Forcing to hold some equalities in  \eqref{eq:lema_nefdiv_cond_L1} and taking Fact 1 into account, one gets
\begin{equation*}
\begin{array}{l}
0=[D]\cdot [\tilde{F}_{j_h}]=b-z_{j_h\,k_h}\displaystyle\sum_{\lambda=1}^{k_h}\mult_{p_{j_h\,\lambda}}(\varphi_{F_{j_h}})\mult_{p_{j_h\,\lambda}}(\varphi_{j_h\,k_h}), \text{ for } 1\leq h\leq s, \text{ and}\\[2mm]
0=[D]\cdot [\tilde{M}_{0}]=a-\displaystyle\sum_{h=1}^s z_{j_h\,k_h}\displaystyle\sum_{\lambda=1}^{k_h}\mult_{p_{j_h\,\lambda}}(\varphi_{M_0})\mult_{p_{j_h\,\lambda}}(\varphi_{j_h\,k_h}).
\end{array}
\end{equation*}
Finally, by the Noether formula, one obtains the following equalities
$$
\begin{array}{l}
a=\displaystyle\sum_{h=1}^s z_{j_h\,k_h}(\varphi_{M_0},\varphi_{j_h\,k_h})_{p_{j_h\,1}}, \text{ and }\\
b=z_{j_h\,k_h}(\varphi_{F_{j_h}},\varphi_{j_h\,k_h})_{p_{j_h\,1}}, \text{ independently of the value   }h, 1\leq h \leq s,
\end{array}
$$
giving rise to the divisor $W_{(j_{1}\,k_{1}),\ldots,(j_{s}\,k_{s})}$ defined in the statement (whose class belongs to $\mathfrak{C}^\vee(X)$ by construction). This completes the proof.
\end{proof}

\begin{remark}\label{rem:valores_z_W}
Set $W_{(j_1\,k_1),\ldots,(j_s\,k_s)}$ a divisor as in Lemma \ref{lema:Dual_cone_S} and $(z_{j_{1}\,k_{1}},\ldots,z_{j_{s}\,k_{s}})$ an element in $(\mathbb{Z}_{>0})^s$ involved in the expression of  $W_{(j_1\,k_1),\ldots,(j_s\,k_s)}$. It holds that any other element $(z_{j_{1}\,k_{1}}',\ldots,z_{j_{s}\,k_{s}}')$ in $(\mathbb{Z}_{>0})^s$ satisfying the equalities $z_{j_{1}\,k_{1}}'\cdot b_{j_{1}\,k_{1}}=\cdots =z_{j_{s}\,k_{s}}'\cdot b_{j_{s}\,k_{s}}$ provides a divisor  $W_{(j_1\,k_1),\ldots,(j_s\,k_s)}' = \gamma\cdot W_{(j_1\,k_1),\ldots,(j_s\,k_s)}$, $\gamma\in \mathbb{Q}_{>0}$, and then its class generates the same ray in $\mathfrak{C}^\vee(X)$ as that determined by the class of $W_{(j_1\,k_1),\ldots,(j_s\,k_s)}$. From now on, we consider the \emph{primitive} element $(z_{j_{1}\,k_{1}},\ldots,z_{j_{s}\,k_{s}})\in (\mathbb{Z}_{>0})^s$ satisfying the equalities $z_{j_{1}\,k_{1}}\cdot b_{j_{1}\,k_{1}}=\cdots =z_{j_{s}\,k_{s}}\cdot b_{j_{s}\,k_{s}}$.
\end{remark}

\begin{remark}\label{const}
Notice that, if the configuration $\mathcal{C}$ in Lemma \ref{lema:Dual_cone_S} consists of a unique constellation, then there is no divisor of the type $W_{(j_{1}\,k_{1}),\ldots,(j_{s}\,k_{s})}$.
\end{remark}

To finish this subsection we give two examples. They consider two rational surfaces $X$ as before and, following Lemma \ref{lema:Dual_cone_S},  describe  a set of generators of the cone $\mathfrak{C}^\vee(X)$. Our first example, Example \ref{New}, is simpler than the second one, Example \ref{example_1}. The first one allows us to fully show how we perform our computations to get the generators of the mentioned cone, while the second one manifests the complexity of $\mathfrak{C}^\vee(X)$ even when one has a relatively small number of blowups. Later, both examples will be continued in Examples \ref{New2} and \ref{example_2}. Note that our final goal is to give bounds on $\delta$ implying the finite polyhedrality with minimal generation of the cone of curves NE$(X)$. In the first example, the finite polyhedrality of NE$(X)$ was known since we blow up less than eight points and, then, $K_X^2 > 0$, $K_X$ being a canonical divisor of $X$ \cite[Theorem 2]{GalMon1}. The reason to include Examples \ref{New} and \ref{New2} is that they are simple enough to explicitly describe our results.

\begin{example}\label{New}
Let $\pi:X\to\mathbb{F}_\delta$ be a sequence of blowups over a Hirzebruch surface at a configuration $\mathcal{C}=\{p_i\}_{i=1}^6$ whose arrowed proximity graph is depicted in Figure  \ref{fig:sencillo}.
\begin{figure}[h!]
        \centering
        \includegraphics[scale=0.8]{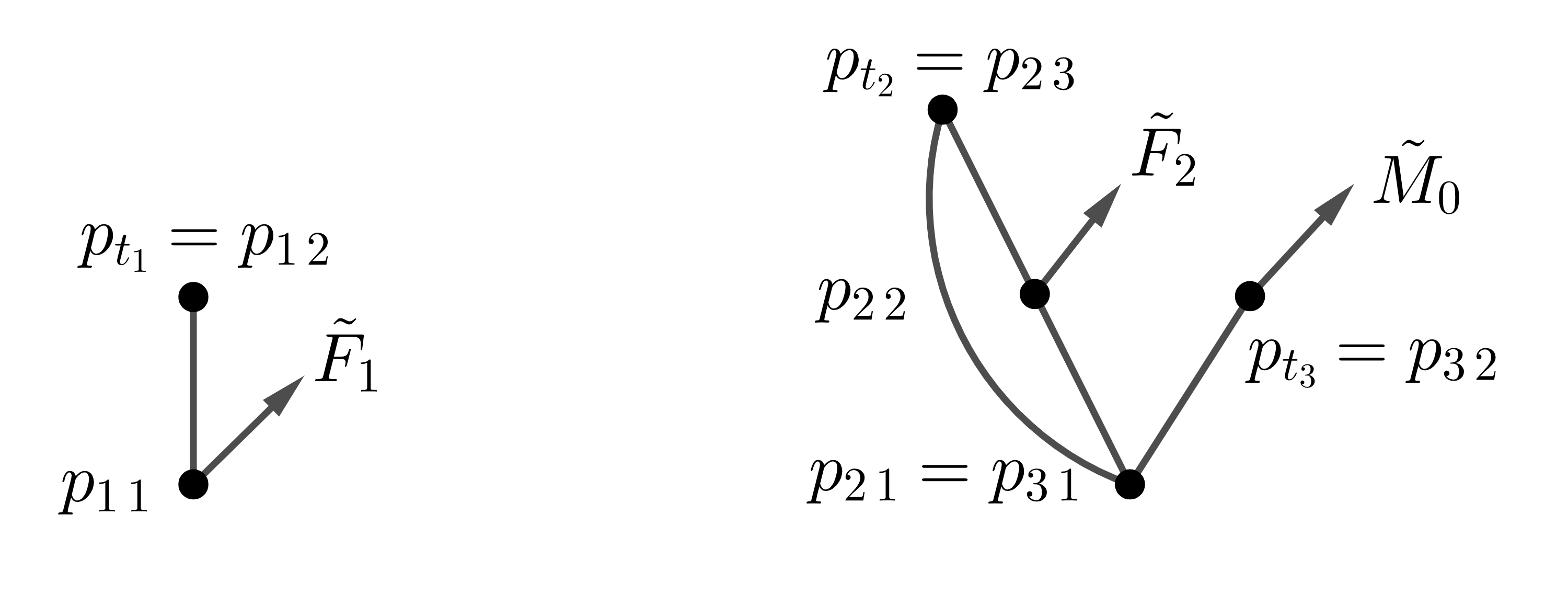}
        \caption{Arrowed proximity graph in Example \ref{New}.}
        \label{fig:sencillo}
\end{figure}
The configuration $\mathcal{C}$ is a disjoint union of two constellations $\mathcal{C}_1$ and $\mathcal{C}_2$ from the left to the right. Here the convex cone $\mathfrak{C}(X)$ introduced at the beginning of this section is spanned by the set of classes $\{[\tilde{F}_1],[\tilde{F}_2],[\tilde{M}_0],[E_1],\ldots,$ $[E_{6}]\},$ where $F_i$ is the fiber passing through the origin of the constellation $\mathcal{C}_i,i=1,2,$ and $M_0$ the special section. The maximal points of $\mathcal{C}$ are marked with a label $t_j,1\leq j\leq 3.$ Let $\nu_j:=\nu_{E_{p_{t_j}}}$ be the divisorial valuation defined by the exceptional divisor $E_{p_{t_j}},1\leq j\leq 3.$ Hence, it holds that
$$
\begin{array}{c}
F_{\nu_1}=F_1,\,F_{\nu_2}=F_2=F_{\nu_3},\, E_{2\,1}^*=E_{3\,1}^*,\\[2mm]
\tilde{F}_1\sim F_1^*-E_{1\,1}^*,\, \tilde{F}_2\sim F_2^*-E_{2\,1}^*-E_{2\,2}^*,\,\tilde{M}_0\sim M_0^*-E_{3\,1}^*-E_{3\,2}^*,
\end{array}
$$
where $\sim$ means linear equivalence. In addition, $N=6,n_1=2,n_2=3,n_3=2$ and $s=2$.

To facilitate the understanding of our results, we compute the explicit expression of the divisors indicated in Lemma \ref{lema:Dual_cone_S}. To do it, we need to calculate some values. The proximity equalities can be deduced from the arrowed proximity graph. They allow us to obtain the sets of multiplicities of the analytically irreducible germs $\varphi_{jk}$ introduced in Section \ref{Sect3}, which are the following ones:
$$
\begin{array}{c}
(\mult_{p_{1\,\lambda}}(\varphi_{1\,1}))_{\lambda=1}^2=(1,0)=(\mult_{p_{3\,\lambda}}(\varphi_{3\,1}))_{\lambda=1}^2,\\[2mm]
(\mult_{p_{1\,\lambda}}(\varphi_{1\,2}))_{\lambda=1}^2=(1,1)=(\mult_{p_{3\,\lambda}}(\varphi_{3\,2}))_{\lambda=1}^2,\\[2mm]
(\mult_{p_{2\,\lambda}}(\varphi_{2\,1}))_{\lambda=1}^3=(1,0,0), \;
(\mult_{p_{2\,\lambda}}(\varphi_{2\,2}))_{\lambda=1}^3=(1,1,0),\\[2mm]
(\mult_{p_{2\,\lambda}}(\varphi_{2\,3}))_{\lambda=1}^3=(2,1,1).\\[2mm]
\end{array}
$$
By using the Noether formula, one gets
$$
\begin{array}{c}
a_{1\,1}=(\varphi_{M_0},\varphi_{1\,1})_{p_{1\,1}}=0=a_{1\,2}=(\varphi_{M_0},\varphi_{1\,2})_{p_{1\,1}}, \ a_{2\,1}=(\varphi_{M_0},\varphi_{2\,1})_{p_{2\,1}}=1=a_{2\,2}, \\[2mm]
a_{3\,1}=(\varphi_{M_0},\varphi_{3\,1})_{p_{3\,1}}=1,\ a_{3\,2}=(\varphi_{M_0},\varphi_{3\,2})_{p_{3\,1}}=2=a_{2\,3}=(\varphi_{M_0},\varphi_{2\,3})_{p_{2\,1}},\\[2mm]
 b_{1\,1}=(\varphi_{F_{1}},\varphi_{1\,1})_{p_{1\,1}}=1= \ b_{1\,2}=(\varphi_{F_{1}},\varphi_{1\,2})_{p_{1\,1}},\\[2mm]
 b_{2\,1}=(\varphi_{F_{2}},\varphi_{2\,1})_{p_{2\,1}}=1,b_{2\,2}=(\varphi_{F_{2}},\varphi_{2\,2})_{p_{2\,1}}=2, b_{2\,3}=(\varphi_{F_{2}},\varphi_{2\,3})_{p_{2\,1}}=3,   \\[2mm] b_{3\,1}=(\varphi_{F_{2}},\varphi_{3\,1})_{p_{3\,1}}=1=b_{3\,2}=(\varphi_{F_{2}},\varphi_{3\,2})_{p_{3\,1}}.\\
 \end{array}
$$
Consequently, by Lemma \ref{lema:Dual_cone_S}, we get the following set $\mathcal{G}$ of divisors whose classes in NS$_\mathbb{R}(X)$ generate the cone $\mathfrak{C}^\vee(X)$: $F^*,M^*,$
$$
\begin{array}{rl}
\Lambda_{1\,1} =& M^*-E_{1\,1}^*.\\[2mm]
\Lambda_{1\,2} =& M^*-E_{1\,1}^*-E_{1\,2}^*.\\[2mm]
\Lambda_{2\,1} =& F^*+ M^*-E_{2\,1}^*=F^*+ M^*-E_{3\,1}^*=\Lambda_{3\,1}.\\[2mm]
\Lambda_{2\,2} =& F^*+ 2M^*-E_{2\,1}^*-E_{2\,2}^*.\\[2mm]
\Lambda_{2\,3} =& 2F^*+ 3M^*-2E_{2\,1}^*-E_{2\,2}^*-E_{2\,3}^*.\\[2mm]
\Lambda_{3\,2} =& 2F^*+ M^*-E_{3\,1}^*-E_{3\,2}^*.\\[2mm]
\end{array}
$$
$$
\begin{array}{rl}
W_{(1\,1),(2\,1)} = &z_{2\,1}F^*+z_{1\,1}M^* -z_{1\,1}E_{1\,1}^*-z_{2\,1}E_{2\,1}^*, \\
 & \text{where }z_{1\,1}\text{ and }z_{2\,1}\text{ are any two positive integers such that}\\
& z_{1\,1}b_{1\,1}=z_{2\,1}b_{2\,1} \mbox{ (that is, $z_{1\,1}=z_{2\,1}$)}. \text{ According to Remark \ref{rem:valores_z_W},}\\
 & \text{we consider $z_{1\,1}=1=z_{2\,1}$. Then our choice is}\\ &  W_{(1\,1),(2\,1)} = F^*+M^* -E_{1\,1}^*-E_{2\,1}^*. \\[1.15mm]
& \text{We only show our choice for the following divisors.}\\[1.15mm]
W_{(1\,1),(3\,1)} = &F^*+M^* -E_{1\,1}^*-E_{3\,1}^*=W_{(1\,1),(2\,1)}.\\[2mm]
W_{(1\,2),(2\,1)} = &F^*+M^* -E_{1\,1}^*-E_{1\,2}^*-E_{2\,1}^*=F^*+M^* -E_{1\,1}^*-E_{1\,2}^*-E_{3\,1}^*=W_{(1\,2),(3\,1)}.\\[2mm]
W_{(1\,1),(3\,2)} =&2F^*+M^* -E_{1\,1}^*-E_{3\,1}^*-E_{3\,2}^*.\\[2mm]
W_{(1\,2),(3\,2)} =&2F^*+M^* -E_{1\,1}^*-E_{1\,2}^*-E_{3\,1}^*-E_{3\,2}^*.\\[2mm]
W_{(1\,1),(2\,2)} =& F^*+2M^* -2E_{1\,1}^*-E_{2\,1}^*-E_{2\,2}^*\\[2mm]
 W_{(1\,2),(2\,2)} =& F^*+2M^* -2E_{1\,1}^*-2E_{1\,2}^*-E_{2\,1}^*-E_{2\,2}^*.\\[2mm]
 W_{(1\,1),(2\,3)} =& 2F^*+3M^* -3E_{1\,1}^*-2E_{2\,1}^*-E_{2\,2}^*-E_{2\,3}^*.\\[2mm]
 W_{(1\,2),(2\,3)} =& 2F^*+3M^* -3E_{1\,1}^*-3E_{1\,2}^*-2E_{2\,1}^*-E_{2\,2}^*-E_{2\,3}^*. \\[1.15mm]
\end{array}
$$
\end{example}

\begin{example}\label{example_1}
Let $\pi:X\to\mathbb{F}_\delta$ be the sequence of blowups over a Hirzebruch surface $\mathbb{F}_\delta$ considered in Example \ref{ex:grafo_proximidad}. The configuration $\mathcal{C}$  in that example is a disjoint union of three constellations $\mathcal{C}_1,\mathcal{C}_2$ and $\mathcal{C}_3$ depicted in Figure \ref{fig_conf_F_delta} from the left to the right. In this example the convex cone $\mathfrak{C}(X)$  is gene\-rated by the set of classes $\{[\tilde{F}_1],[\tilde{F}_2],[\tilde{F}_3],[\tilde{M}_0],[E_1],$ $\ldots,[E_{20}]\},$ where $F_i$ is the fiber which goes through the origin of the constellation $\mathcal{C}_i,1\leq i\leq 3,$ and $M_0$ the special section. According to our notation, we mark with a label $t_j, 1\leq j\leq 5,$ the maximal points of $\mathcal{C}$. Let $\nu_j:=\nu_{E_{p_{t_j}}} $ be the divisorial valuation defined by the exceptional divisor $E_{p_{t_j}},1\leq j\leq m= 5.$ Then:
\begin{equation*}
\begin{array}{c}
F_{\nu_1}=F_1=F_{\nu_2},\ F_{\nu_3}=F_2=F_{\nu_4},\ F_{\nu_5}=F_3,\\[2mm]
E_{1\,1}^*=E_{2\,1}^*, \ E_{1\,2}^*=E_{2\,2}^*, \ E_{1\,3}^*=E_{2\,3}^*, \ E_{3\,1}^*=E_{4\,1}^*,  \\[2mm]
\tilde{F}_1 \sim F_1^*-E_{1\,1}^*, \ \tilde{F}_2 \sim F_2^*-E_{3\,1}^* - E_{3\,2}^*, \ \tilde{F}_3 \sim F_3^*-E_{5\,1}^*-E_{5\,2}^*,\\[2mm]
\tilde{M}_0 \sim M_0^* -E_{1\,1}^* -E_{1\,2}^* -E_{1\,3}^* -E_{1\,4}^* -E_{1\,5}^* -E_{4\,1}^* -E_{4\,2}^* -E_{4\,3}^*.
\end{array}
\end{equation*}
Moreover, $N=20,n_1=7,n_2=5,n_3=4,n_4=5,n_5=3$ and $1< s \leq r=3.$

In this example there are $361$  divisors (introduced in Lemma \ref{lema:Dual_cone_S}) that span the cone $\mathfrak{C}^\vee(X)$. For this reason and to illustrate our result, we only give some of them. They are the following ones:
$$
\begin{array}{c}
\Lambda_{1\,3}\,(j=1,k=3),\Lambda_{1\,7}\,(j=1,k=7),\Lambda_{2\,3}\,(j=2,k=3),\Lambda_{3\,4}\,(j=3,k=4),\\
\Lambda_{5\,3}\,(j=5,k=3), W_{(1\,7),(3\,4)}\, (j_1=1,k_1=7,j_2=3,k_2=4) , \\
W_{(3\,4),(5\,3)}\, (j_1=3,k_1=4,j_2=5,k_2=3)\text{ and }\\
W_{(1\,7),(3\,4),(5\,3)}\,(j_1=1,k_1=7,j_2=3,k_2=4,j_3=5,k_3=3).
\end{array}
$$
The strict transform of the (analytically irreducible) germ $\varphi_{j\,k}$ at the point $p_{j\,1}$ is transversal to the exceptional divisor $E_{j\,k}$ at a general point, for those pairs $(j,k)$ belonging to the set $\{(1,3),(1,7),(2,3),(3,4), (5,3)\}$.  By the proximity equalities, we obtain the following sets of multiplicities:
$$
\begin{array}{c}
(\mult_{p_{1\,\lambda}}(\varphi_{1\,3}))_{\lambda=1}^7=(1,1,1,0,0,0,0),\,(\mult_{p_{2\,\lambda}}(\varphi_{2\,3}))_{\lambda=1}^5=(1,1,1,0,0),  \\[2mm] (\mult_{p_{1\,\lambda}}(\varphi_{1\,7}))_{\lambda=1}^7=(3,3,3,3,1,1,1),
(\mult_{p_{3\,\lambda}}(\varphi_{3\,4}))_{\lambda=1}^4=(2,2,1,1)  \\[2mm]
\text{and } (\mult_{p_{5\,\lambda}}(\varphi_{5\,3}))_{\lambda=1}^3=(2,1,1).
\end{array}
$$
In addition, by the Noether formula, it holds that
$$
\begin{array}{c}
a_{1\,3}=(\varphi_{M_0},\varphi_{1\,3})_{p_{1\,1}}=3, a_{1\,7}=(\varphi_{M_0},\varphi_{1\,7})_{p_{1\,1}}=13, \ a_{2\,3}=(\varphi_{M_0},\varphi_{2\,3})_{p_{2\,1}}=3, \\ a_{3\,4}=(\varphi_{M_0},\varphi_{3\,4})_{p_{3\,1}}=2,\ a_{5\,3}=(\varphi_{M_0},\varphi_{5\,3})_{p_{5\,1}}=0,\\[2mm]
 b_{1\,3}=(\varphi_{F_{1}},\varphi_{1\,3})_{p_{1\,1}}=1, \ b_{1\,7}=(\varphi_{F_{1}},\varphi_{1\,7})_{p_{1\,1}}=3,\ b_{2\,3}=(\varphi_{F_{1}},\varphi_{2\,3})_{p_{2\,1}}=1  \\ b_{3\,4}=(\varphi_{F_{2}},\varphi_{1\,7})_{p_{3\,1}}=4, \ b_{5\,3}=(\varphi_{F_{3}},\varphi_{5\,3})_{p_{5\,1}}=3.\\
 \end{array}
$$
Thus, one gets
$$
\begin{array}{rl}
\Lambda_{1\,3} =& 3F^*+M^*-E_{1\,1}^*-E_{1\,2}^*- E_{1\,3}^*\\=&3F^*+M^*-E_{2\,1}^*-E_{2\,2}^*- E_{2\,3}^*=\Lambda_{2\,3}.\\[2mm]
\Lambda_{1\,7} =& 13 F^*+3M^* -3E_{1\,1}^*-3E_{1\,2}^*- 3E_{1\,3}^*- 3E_{1\,4}^*-E_{1\,5}^*-E_{1\,6}^*-E_{1\,7}^*.\\[2mm]
\Lambda_{3\,4} = & 2 F^*+4M^* -2E_{3\,1}^*-2E_{3\,2}^*- E_{3\,3}^*- E_{1\,4}^*.\\[2mm]
\Lambda_{5\,3} = & 3M^* -2E_{5\,1}^*-E_{5\,2}^*- E_{5\,3}^*.\\[2mm]
W_{(1\,7),(3\,4)} = & (13z_{1\,7} + 2z_{3\,4}) F^*+3z_{1\,7}M^* -3z_{1\,7}E_{1\,1}^*-3z_{1\,7}E_{1\,2}^*- 3z_{1\,7}E_{1\,3}^* \\ &- 3z_{1\,7}E_{1\,4}^*-z_{1\,7}E_{1\,5}^*-z_{1\,7}E_{1\,6}^*-z_{1\,7}E_{1\,7}^* -2z_{3\,4}E_{3\,1}^*-2z_{3\,4}E_{3\,2}^*\\ &- z_{3\,4}E_{3\,3}^*- z_{3\,4}E_{3\,4}^*, \\
 & \text{ where }z_{1\,7}\text{ and }z_{3\,4}\text{ are any two positive integers such that } 3z_{1\,7}=4 z_{3\,4}.\\
& \text{ According to Remark \ref{rem:valores_z_W}, we consider $z_{1\,7}=4$ and $z_{3\,4}=3$. Then our}\\ & \text{ choice for } W_{(1\,7),(3\,4)} \text{ is } \\[1.15mm]
W_{(1\,7),(3\,4)} = & 58 F^*+12M^* -12 E_{1\,1}^*-12 E_{1\,2}^*- 12 E_{1\,3}^* - 12 E_{1\,4}^*-4 E_{1\,5}^*-4 E_{1\,6}^*\\ &-4E_{1\,7}^* -6 E_{3\,1}^*-6 E_{3\,2}^*- 3 E_{3\,3}^*- 3 E_{3\,4}^*.\\[1.5mm]
 &  \text{In the next two descriptions, we only show a choice.}\\[1.5mm]
W_{(3\,4),(5\,3)} = & 6 F^*+12M^* -6E_{3\,1}^*-6E_{3\,2}^*- 3E_{3\,3}^*- 3E_{3\,4}^* -8E_{5\,1}^*-4E_{5\,2}^*- 4E_{5\,3}^*.\\[2mm]
W_{(1\,7),(3\,4),(5\,3)}=& 58 F^*+12M^* -12E_{1\,1}^*-12E_{1\,2}^*- 12E_{1\,3}^* - 12E_{1\,4}^*-4E_{1\,5}^*-4E_{1\,6}^*\\ &-4E_{1\,7}^* -6E_{3\,1}^*-6E_{3\,2}^*- 3E_{3\,3}^*- 3E_{3\,4}^*-8E_{5\,1}^*-4E_{5\,2}^*- 4E_{5\,3}^*. \\[2mm]
\end{array}
$$
\end{example}

\subsection{The values $a(\operatorname{APG}(X,\pi))$ and $b'(\operatorname{APG}(X,\pi))$} \label{42}
In this subsection we introduce two nonnegative integers, which depend only on the arrowed proximity graph, and are useful in our main result. These values are denoted  by $a(\operatorname{APG}(X,\pi))$ and $b'(\operatorname{APG}(X,\pi))$. We will prove that they can be algorithmically obtained and, in particular cases, given by explicit formulae.

\begin{proposition}\label{thm:nonnegative_selfinters}
Let $X$ be a rational surface which is obtained by a composition $\pi:X\to \mathbb{F}_\delta$ of blowups at a configuration of infinitely near points over $\mathbb{F}_{\delta}$. There exists a positive integer $a(\operatorname{APG}(X,\pi)),$ which depends only on the arrowed proximity graph $\operatorname{APG}(X,\pi),$ such that, if $\delta\geq a(\operatorname{APG}(X,\pi))$, then the generators of the dual cone $\mathfrak{C}^\vee(X)$ computed in Lemma \ref{lema:Dual_cone_S} have nonnegative self-intersection.
\end{proposition}
\begin{proof}
The coefficients of the divisors shown in Lemma \ref{lema:Dual_cone_S} do not depend on the value of $\delta$. In addition, any of these divisors can be written as
$$
    D:=aF^*+bM^*-\sum_{l=1}^N c_l E_l^*,
$$
depending only on the $\operatorname{APG}(X,\pi),$ and then
$$
D^2=2ab+b^2\delta - \sum_{l=1}^N c_l^2.
$$
Therefore, it suffices to look for the lowest positive integer $a(\operatorname{APG}(X,\pi))$ such that, when $\delta\geq a(\operatorname{APG}(X,\pi))$, all the finitely many divisors defined in Lemma \ref{lema:Dual_cone_S} have nonnegative self-intersection.
\end{proof}

\begin{remark}
\label{GetA}
Given a rational number $x$, we set $\lceil x \rceil^*$ the minimum positive integer being an upper bound of $\{x\}$. The proof of Proposition \ref{thm:nonnegative_selfinters} shows that
\[
a(\operatorname{APG}(X,\pi)) = \max \left\{ \Bigg\lceil\dfrac{\sum_{l=1}^N c_l^2-2ab}{b^2}\Bigg\rceil^* : D = aF^*+bM^*-\sum_{l=1}^N c_l E_l^* \in  \tilde{\mathcal{G}}:= \mathcal{G}\setminus\{F^*\} \right\},
\]
where $\mathcal{G}$ is the finite set of divisors given in Lemma \ref{lema:Dual_cone_S}. Recall that the classes of the divisors in $\mathcal{G}$ span the cone $\mathfrak{C}^\vee(X)$ and the divisors $D \in \tilde{\mathcal{G}}$ have nonvanishing coefficient $b$.
\end{remark}

Let $X$ be a rational surface coming from a composition $\pi:X\to \mathbb{F}_\delta$ as in Proposition \ref{thm:nonnegative_selfinters}. As we said in Section \ref{Sect3}, the labeled graph obtained by deleting the arrows of $\operatorname{APG}(X,\pi)$ is named \emph{proximity graph} of $X$ and it is denoted by $\operatorname{PG}(X,\pi)$. When the ground field is $\mathbb{C}$, it can be said that the information encoded by $\operatorname{PG}(X,\pi)$ is ``topological'' in the following sense. Suppose that $\mathcal{C}$ is a disjoint union of the \emph{constellations} ${\mathcal C}_1,\ldots,{\mathcal C}_r$ and, for each $i=1,\ldots,r$,  set $\mathcal{C}_i=\cup_{j=j_{i-1}+1}^{j_i} {\mathcal C}_{\nu_{j}}$ for suitable indices $0=j_0<j_1<\cdots <j_r=m$, the decomposition of $\mathcal{C}_i$ as union of chains attached to the valuations $\nu_{j}$ defined by the final exceptional divisors (as in \eqref{val}). Let $\phi_{j}$ be a general element of the valuation $\nu_{j}$ and consider the germ at the origin of $\mathcal{C}_i$ in $\mathbb{F}_\delta$ given by $\psi_i:=\prod_{j=j_{i-1}+1}^{j_i} \phi_{j}$. Then the proximity graph $\operatorname{PG}(X,\pi)$ depends only on $(S_1,\ldots, S_r)$, where $S_i$ denotes the topological type of the singularity defined by $\psi_i$. Our next result shows that the ``topology'' is not enough to determine the value $a(\operatorname{APG}(X,\pi))$ provided by Proposition \ref{thm:nonnegative_selfinters}. We also prove that $a(\operatorname{APG}(X,\pi))$ can be explicitly computed in some cases.

\begin{corollary}\label{cor:delta0}
Keep the notation introduced previously.
\begin{itemize}
\item[(a)] Consider a family $\mathcal{F}$ of rational surfaces $X$ formed by a composition of blowups $\pi:X\to \mathbb{F}_\delta$ as in Proposition \ref{thm:nonnegative_selfinters}, but having the same proximity graph $\operatorname{PG}(X,\pi)$. Then $\{a(\operatorname{APG}(X,\pi))\,|\, X \in \mathcal{F}\}$ is a finite set whose values depend not only on $\operatorname{PG}(X,\pi)$ but on the number of (necessarily free) points of the configuration $\mathcal{C}$ (given by $\pi$) through which the strict transforms of the fibers and the special section of $\mathbb{F}_\delta$ pass.
\item[(b)] Suppose that the configuration $\mathcal{C}$ has only free points and assume that, for any divisorial valuation $\nu_j$ defined by a final exceptional divisor of $\mathcal{C}$, $\nu_j(\varphi_{M_0})=0$ and $\nu_j(\varphi_{F_j})=1$, $F_j$ being the fiber of $\mathbb{F}_\delta$ going through the origin of $\mathcal{C}_{\nu_j}$. Then,
    $$
    a(\operatorname{APG}(X,\pi))=\max\Bigg\lbrace\sum_{h=1}^s \# (\mathcal{C}_{\nu_{j_h}}) \ \bigg| \  {\footnotesize \begin{array}{l}
       \text{$\big\lbrace \mathcal{C}_{\nu_{j_h}} \big\rbrace_{h=1}^s$ is a set of pairwise}\\ \text{disjoint chains whose origins} \\ \text{belong to different fibers}\\
    \end{array}}\Bigg\rbrace,
    $$
  where $\# (\mathcal{C}_{\nu_{j_h}})$ stands for the cardinality of $\mathcal{C}_{\nu_{j_h}}$.

    \item[(c)] Suppose that $\mathcal{C}$ is constellation. Then,
    $$a(\operatorname{APG}(X,\pi))=\max\Bigg\lbrace \Bigg\lceil\dfrac{\sum_{\lambda=1}^{n_{j}}\mult_{p_{j\,\lambda}}(\varphi_{j\,n_{j}})^2-2a_{j\,n_j}\,b_{j\,n_j}}{b_{j\,n_j}^2}\Bigg\rceil^* \ \bigg| \  1\leq j\leq m
\Bigg\rbrace,
     $$
where $\lceil x \rceil^*$ is as defined in Remark \ref{GetA}.
\end{itemize}

\end{corollary}

\begin{proof}
We start by showing Statement (a). As we said in the proof of Proposition \ref{thm:nonnegative_selfinters}, any divisor $D$ considered in Lemma \ref{lema:Dual_cone_S} can be written as $D=aF^*+bM^*-\sum_{l=1}^N c_l E_l^*$  and satisfies $
D^2=2ab+b^2\delta - \sum_{l=1}^N c_l^2$. By Remark \ref{GetA}, the value $a(\operatorname{APG}(X,\pi))$ introduced in Proposition \ref{thm:nonnegative_selfinters} is the maximum of the values
\begin{equation}\label{eq:Cor_delta0}
\Bigg\lceil\dfrac{\sum_{l=1}^N c_l^2-2ab}{b^2}\Bigg\rceil^*,
\end{equation}
where $aF^*+bM^*-\sum_{l=1}^N c_l E_l^*$ runs over $\mathcal{G}\setminus\{F^*\}$, $\mathcal{G}$ being the set of divisors provided in Lemma \ref{lema:Dual_cone_S}.

By Lemma \ref{lema:Dual_cone_S}, the values $a$ are a sum of nonnegative integers depending on local intersection pro\-ducts as $(\varphi_{M_0},\varphi)_p$, $p$ being an origin of the configuration $\mathcal{C}$ defined by $\pi$ and $\varphi$ a suitable analytically irreducible germ at $p$. The values $b$ are positive integers depending on local intersection products as $(\varphi_{F},\varphi)_p$, where $F$ is the fiber going through $p.$ By the Noether formula and since $X \in \mathcal{F}$, the values  $(\varphi_{M_0},\varphi)_p$ and $(\varphi_{F},\varphi)_p$ depend on the number of points in $\mathcal{C}$ through which the strict transforms of the fibers $F_i$ and the special section $M_0$ go, values which are determined by the arrows in $\operatorname{APG}(X,\pi)$. Notice that $F_i$ and $M_0$ go through free points and there are finitely many possibilities for the positions of the arrows. This proves the first statement because when $X\in\mathcal{F}$ only the products $(\varphi_{F},\varphi)_p$ can modify the values $a$, $b$ and $c_l$.

Let us prove (b). As the points $p_l$ in $\mathcal{C}$ are free, the nonvanishing coefficients $c_l$ of the divisors $D$ in Lemma \ref{lema:Dual_cone_S} are equal to $1$. Hence, any of these divisors $D$ with nonempty exceptional support has the form $M^*-\sum_{i=1}^N  c_l E_l^*$, where $c_l$ equals either $0$ or $1$. As a consequence, the value $a(\operatorname{APG}(X,\pi))$ is attained by some divisors $W_{(j_1,n_{j_1}),\ldots,(j_s,n_{j_s})}$ introduced in Lemma \ref{lema:Dual_cone_S}. This is because the support of this type of divisors of $\mathfrak{C}^\vee(X)$ contains the maximum number of exceptional divisors. This completes the proof of (b).

Finally, we show (c). By Lemma \ref{lema:Dual_cone_S} and Remark \ref{const}, $\mathfrak{C}^\vee(X)$ is just generated by the classes of the divisors $F^*,M^*$ and $\Lambda_{j\,k}$. Moreover, according to \cite[Lemma 3.4]{GalMonMor}, if the divisor $\Lambda_{j\,n_j}$ satisfies $\Lambda_{j\,n_j}^2\geq 0$ then $\Lambda_{j\,k}^2\geq 0$ for all $k\in\{1,\ldots, n_{j}-1\}$. As a consequence, computing $a(\operatorname{APG}(X,\pi))$ is reduced to find the lowest positive integer such that all the divisors $\Lambda_{j\,n_j},1\leq j\leq m,$ have nonnegative self-intersection, which finishes the proof.
\end{proof}

\begin{remark}\label{re:calculo_a_apartado_b_corolario}
Let $\mathcal{C}$ be a configuration as in Statement (b) of Corollary \ref{cor:delta0}. To get the maximum giving rise to $a(\operatorname{APG}(X,\pi))$, one does not need to consider all the valuations $\nu_{j_h}$. Within each constellation of $\mathcal{C}$ and among their chains given by the final exceptional divisors, it suffices to take that associated to a valuation $\nu_{j_{h'}}$ with largest value $n_{j_{h'}}$. Then, $a(\operatorname{APG}(X,\pi))$ can be computed with the same formula as in Statement (b) but applied to this smaller set of valuations we have just described.
\end{remark}

Let $-K_{\mathbb{F}_\delta}$ be an anticanonical divisor on the Hirzebruch surface $\mathbb{F}_\delta$. By \cite[Chapter V, Section 2, Corollary 2.11]{Har},  $-K_{\mathbb{F}_\delta}$ is linearly equivalent to the divisor $\left(2-\delta\right)F+2M$. Set $X$ a rational surface obtained by a composition of blowups $\pi:X\to\mathbb{F}_\delta$ given by a configuration ${\mathcal C} = \{ p_l\}_{l=1}^{N}$. Then, by \cite[Proposition 1.1.26]{Alb}, any anticanonical divisor $-K_X$ on the surface $X$ is linearly equivalent to the divisor $-K_{\mathbb{F}_\delta}^*-\sum_{i=1}^N E_i^*$.
Our next result involves this last divisor.

\begin{proposition}\label{prop:K_Xposi_inter}
Let $X$ be a rational surface as in Proposition \ref{thm:nonnegative_selfinters}. Then, there exists a nonnegative integer $b'(\operatorname{APG}(X,\pi))$, which depends only on the arrowed proximity graph $\operatorname{APG}(X,\pi)$, such that, if $\delta\geq b'(\operatorname{APG}(X,\pi))$, then the intersection product of the class of $-K_X$ with any generator of $\mathfrak{C}^\vee(X)$ is positive.
\end{proposition}
\begin{proof}
As we mentioned before, the coefficients of the divisors shown in Lemma \ref{lema:Dual_cone_S} do not depend on the value of $\delta$ and these divisors can be expressed as $D=aF^*+bM^*-\sum_{l=1}^N c_l E_l^*$. Thus,
\begin{equation}\label{Eq_proposition_bX}
D\cdot (- K_{X}) = 2a + b(2 + \delta)-\sum_{l=1}^N c_l.
\end{equation}
Since we have finitely many divisors $D$, the lowest value of $\delta$ such that all the intersection products $D\cdot (- K_{X})$ are positive defines $b'(\operatorname{APG}(X,\pi))$.
\end{proof}

\begin{remark}
\label{GetBp}
Given a rational number $x$, we define $\lceil x \rceil^+$ the ceil of $x$ if $x\geq 0$ and $0$ otherwise.  The proof of Proposition \ref{prop:K_Xposi_inter} shows that
\[
b'(\operatorname{APG}(X,\pi)) = \max \left\{ \Bigg\lceil\dfrac{\sum_{l=1}^N c_l-2a}{b} - 2\Bigg\rceil^+ \;:\; D = aF^*+bM^*-\sum_{l=1}^N c_l E_l^* \in  \mathcal{G}\setminus\{F^*\}  \right\},
\]
where $\mathcal{G} $ is the finite set of divisors given in Lemma \ref{lema:Dual_cone_S} whose classes generate the cone $\mathfrak{C}^\vee(X)$.
\end{remark}

Our next step is to prove that the \emph{value} $b'(\operatorname{APG}(X,\pi))$ defined in the above proof satisfies similar properties to those described for $a(\operatorname{APG}(X,\pi))$ in Corollary \ref{cor:delta0}. We start by stating two lemmas.

\begin{lemma}\label{lem:_seq_valo_betabarras}
Let $\nu$ be a divisorial valuation of $\mathbb{F}_\delta$ and $\mathcal{C_\nu}=\{p_l\}_{l=1}^n$ its configuration of infinitely near points. Set $(\nu(\mathfrak{m}_l))_{l=1}^n$ the sequence of values of $\nu$ and $(\overline{\beta}_w)_{w=0}^{g+1}$ the sequence of maximal contact values of $\nu$. Then it holds that
$$
\sum_{l=1}^n\nu(\mathfrak{m}_l)= \overline{\beta}_{g+1}+\sum_{w=0}^{g}\overline{\beta}_w(1-N_w)+\overline{\beta}_0-1,    
$$
where $N_0=1$ and $N_w=e_{w-1}/e_w,$ with $e_w=\gcd\left(\overline{\beta}_0,\ldots,\overline{\beta}_w\right)$ for $1\leq w\leq g$.

\end{lemma}
\begin{proof}
Set $(\beta_0,\ldots,\beta_{g+1})$ the sequence of characteristic exponents of the divisorial valuation $\nu$ (see \cite[(1.5.3)]{DelGalNun}). It holds that $\beta_0=\overline{\beta}_0$ and
\begin{equation}\label{eq:characteristic_exp_betabarra}
\beta_w=\beta_{w-1}+\overline{\beta}_w-N_{w-1}\overline{\beta}_{w-1}, \text{ for } 1\leq w\leq g+1,
\end{equation}
where $N_0=1.$ In addition, $\overline{\beta}_{g+1}-N_g\overline{\beta}_g$ is the number of free points after the last satellite point of $\mathcal{C}_\nu$. Then, by \cite[Exercise 5.6]{Cas},  one gets that
$$
\sum_{l=1}^n\nu(\mathfrak{m}_l)=\overline{\beta}_{g+1}-N_g\overline{\beta}_g+\beta_g +\beta_0-1.
$$
To obtain the desired formula it suffices to use Equality \eqref{eq:characteristic_exp_betabarra} recursively and the fact that $\beta_0=\overline{\beta}_0$.

\end{proof}

\begin{lemma}\label{lemma_anticanonical_div}
Let $X$ be a rational surface obtained by blowing up at a constellation $\mathcal{C}$ of infinitely near points over a Hirzebruch surface $\mathbb{F}_\delta.$ Set $\Lambda_{j\,k},1\leq k\leq n_j$ and $1\leq j\leq m,$ the divisors defined in Lemma \ref{lema:Dual_cone_S}. Let $-K_{X}$ be an anticanonical divisor on $X.$

 If  $\Lambda_{j\,n_j}\cdot (-K_{X})>0$ holds for all $j\in\{1,\ldots,m\}$, then $\Lambda_{j\,k} \cdot (-K_{X})>0$ for all $k\in\{1,\ldots,n_j\}$ and $j\in\{1,\ldots,m\}$.
\end{lemma}
\begin{proof}
To prove the result is enough to see that, for each $j\in\{1,\ldots,m\}$ and $k\in\{2,\ldots,n_j\}$,  $\Lambda_{j \,k}\cdot (-K_{X})>0$ implies that $\Lambda_{j\, k-1}\cdot (-K_{X})>0.$ Also, without loss of generality, we can assume that $k=n=n_j$ and, simplifying the notation, $\nu_{n}$ (respectively, $\nu_{n-1}$) denotes the divisorial valuation defined by the exceptional divisor $E_{j\,n_j}$ (respectively, $E_{j\,n_j-1}$). By Subsection \ref{subsec:div_val}, the divisor $\Lambda_\tau :=\Lambda_{j\, \tau},\tau=n,n-1$, can be written as
 $$
 \Lambda_\tau=\nu_\tau(\varphi_{M_0})F^*+\nu_\tau(\varphi_{F_{\nu_n}})M^* - \sum_{\lambda=1}^\tau \nu_\tau(\mathfrak{m}_\lambda)E_\lambda^*.
 $$
As a result, one obtains that
\begin{equation}
\Lambda_\tau\cdot (-K_{X})=2\nu_\tau(\varphi_{M_0})+\nu_\tau(\varphi_{F_{\nu_n}})(2+\delta) - \sum_{\lambda=1}^\tau\nu_\tau(\mathfrak{m}_\lambda).
\end{equation}

In order to prove the result, we distinguish two cases: either the point $p_n:=p_{j\,n_j}$ is free or it is a satellite point.

If the point $p_n$ is free, then $\Lambda_{n-1}\cdot (-K_{X})>\Lambda_n\cdot (-K_{X})$ and we obtain the result.

Now let us suppose that $p_n$ is a satellite point. Denote by $(\overline{\beta}_w(\nu_n))_{w=0}^g$ (respectively, $(\overline{\beta}_w(\nu_{n-1}))_{w=0}^{\hat{g}}$) the maximal contact values of $\nu_n$ (respectively, $\nu_{n-1}$). Write $e_w(\nu_\tau)=\gcd\left(\overline{\beta}_0(\nu_\tau),\ldots,\overline{\beta}_w(\nu_\tau)\right),$ $0\leq w\leq \mu,$ and $N_w(\nu_\tau)=e_{w-1}(\nu_\tau)/e_w(\nu_\tau),1\leq w\leq \mu,$ for $\mu=g,\hat{g}$ and $\tau=n,n-1$.

Here, we distinguish two subcases: $\hat{g}=g$ and $\hat{g}=g-1.$

Firstly, consider the \emph{subcase $\hat{g}=g$}  and  set $e= e_{g-1}(\nu_{n-1})/e_{g-1}(\nu_n).$ Then the points $p_n$ and $p_{n-1}$ are sate\-llite. In addition, it holds that
\begin{equation}\label{eq:Lemma_cano_betabarra_mismo_g_1}
\begin{array}{c}
\nu_{n-1}(\varphi_{M_0})=e\nu_{n}(\varphi_{M_0}),\nu_{n-1}(\varphi_{F_{\nu_n}})=e\nu_{n}(\varphi_{F_{\nu_n}}), \overline{\beta}_{g+1}(\nu_\tau)=N_g(\nu_\tau)\overline{\beta}_g(\nu_\tau),\\[2mm]
\overline{\beta}_w(\nu_{n-1})=e\overline{\beta}_{w}(\nu_n) \text{ and } N_w(\nu_{n-1})=N_w(\nu_n),\text{ for }0\leq w < g,
\end{array}
\end{equation}
 and, by \cite[Lemma 2]{GalMon},
\begin{equation}\label{eq:Lemma_cano_betabarras_mismo_g_2}
-\dfrac{\overline{\beta}_g(\nu_{n-1})}{e}\geq -\dfrac{1}{e\,e_{g-1}(\nu_n)}-\overline{\beta_g}(\nu_n).
\end{equation}
As a consequence, one gets that
$$
\begin{array}{rcl}
\Lambda_{n-1}\cdot(-K_{X})&=&e\left(2\nu_{n}(\varphi_{M_0})+\nu_{n}(\varphi_{F_{\nu_n}})\left(2+\delta\right)-\left(\overline{\beta}_{g}(\nu_{n-1})/e\right.\right. \\ & & \hspace{4mm} \left.\left.+\sum_{w=1}^{g-1}\overline{\beta}_w(\nu_{n})(1-N_w(\nu_{n}))+\overline{\beta}_0(\nu_{n})\right)\right)+1\\[2mm]
&\geq & e\left(2\nu_{n}(\varphi_{M_0})+\nu_{n}(\varphi_{F_{\nu_n}})\left(2+\delta\right)-\left(1/(e\,e_{g-1}(\nu_n))+\overline{\beta_g}(\nu_n)\right.\right. \\ & & \hspace{4mm} \left.\left.+\sum_{w=1}^{g-1}\overline{\beta}_w(\nu_{n})(1-N_w(\nu_{n}))+\overline{\beta}_0(\nu_{n})\right)\right)+1\\[2mm]
& = & e \left(\Lambda_{n}\cdot(-K_{X}) -1 \right)- 1/e_{g-1}(\nu_{n-1}) +1\\[2mm]
& > & 0,
\end{array}
$$
where the first equality is deduced from Lemma \ref{lem:_seq_valo_betabarras} and  \eqref{eq:Lemma_cano_betabarra_mismo_g_1}, the first inequality comes from  \eqref{eq:Lemma_cano_betabarras_mismo_g_2}, the second equality follows from \eqref{eq:Lemma_cano_betabarra_mismo_g_1} and the second inequality holds because $\Lambda_{n}\cdot(-K_{X})>0$ and $1/e_{g-1}(\nu_{n-1}) <1$.

Now let us show the \emph{subcase $\hat{g}=g-1$.} Under this situation, the point $p_n$ is satellite and the point $p_{n-1}$ is free. Set $e=1/e_{g-1}(\nu_n).$ Then,
\begin{equation}\label{eq:Lemma_cano_betabarras_mismo_g_3}
\begin{array}{c}
\nu_{n-1}(\varphi_{M_0})=e\nu_{n}(\varphi_{M_0}),\nu_{n-1}(\varphi_{F_\nu})=e\nu_{n}(\varphi_{F_{\nu_n}}), e_{g-1}(\nu_n)=2,\\[2mm]
 \overline{\beta}_w(\nu_{n-1})=e\overline{\beta}_{w}(\nu_n) \text{ and } N_w(\nu_{n-1})=N_w(\nu_n),\text{ for }0\leq w \leq \hat{g}.
\end{array}
\end{equation}
Therefore, we obtain that
$$
\begin{array}{rcl}
\Lambda_{n-1}\cdot(-K_{X})&=&2e\nu_{n}(\varphi_{M_0})+e\nu_{n}(\varphi_{F_{\nu_n}})\left(2+\delta\right)-\left(\overline{\beta}_{\hat{g}+1}(\nu_{n-1})\right. \\ & &  \left.+e\sum_{w=0}^{g-1}\overline{\beta}_w(\nu_{n})(1-N_w(\nu_{n}))+e\overline{\beta}_0(\nu_{n})-1\right)\\[2mm]
& = & e\left(2\nu_{n}(\varphi_{M_0})+\nu_{n}(\varphi_{F_{\nu_n}})\left(2+\delta\right)-\left(\overline{\beta}_{g+1}(\nu_{n})\right.\right. \\ & & \hspace{4mm} \left.\left.+\sum_{w=0}^{g}\overline{\beta}_w(\nu_{n})(1-N_w(\nu_{n}))+\overline{\beta}_0(\nu_{n})-1\right)\right)-\overline{\beta}_{\hat{g}+1}(\nu_{n-1})\\
 & & +e\overline{\beta}_{g+1}(\nu_{n})+e\overline{\beta}_g(\nu_{n})(1-N_g(\nu_{n}))-e+1\\[2mm]
 & = & e\left(\Lambda_{n}\cdot(-K_{X})\right) + e\overline{\beta}_g(\nu_{n})-\overline{\beta}_{\hat{g}+1}(\nu_{n-1})-e+1\\[2mm]
 & = & e\left(\Lambda_{n}\cdot(-K_{X})\right)\\[2mm]
& > & 0,
\end{array}
$$
where the first equality comes from \eqref{eq:Lemma_cano_betabarras_mismo_g_3}, the second one is true because we consider superfluous elements to get the product $\Lambda_{n}\cdot(-K_{X})$ and the last one follows from the equalities $e_{g-1}(\nu_n)=2,$ since $p_n$ is a satellite point and $p_{n-1}$ is a free point, and
$$
\overline{\beta}_{\hat{g}+1}(\nu_{n-1})=\dfrac{e_{g-1}(\nu_{n})\overline{\beta}_g(\nu_{n})+2}{4}.
$$
This completes the proof.
\end{proof}

Now, we are ready to state the announced corollary of Proposition \ref{prop:K_Xposi_inter} showing that the behaviour of the value $b'(\operatorname{APG}(X,\pi))$ is close to that of $a(\operatorname{APG}(X,\pi))$.

\begin{corollary}\label{cor:value_b}
Keep the notation introduced previously.
\begin{itemize}
\item[(a)] Consider a family $\mathcal{F}$ of rational surfaces $X$ formed by a composition of blowups $\pi:X\to \mathbb{F}_\delta$ as in Proposition \ref{thm:nonnegative_selfinters}, but having the same proximity graph $\operatorname{PG}(X,\pi)$. Then, the set $\{b'(\operatorname{APG}(X,\pi)) \, | \, X \in \mathcal{F}\}$ is finite and its values depend not only on $\operatorname{PG}(X,\pi)$ but on the number of (necessarily free) points of the configuration $\mathcal{C}$ (given by $\pi$) through which the strict transforms of the fibers and the special section of $\mathbb{F}_\delta$ pass.

\item[(b)] Suppose that the configuration $\mathcal{C}$ has only free points. Assume that, for any divisorial valuation $\nu_j$ defined by a final exceptional divisor of $\mathcal{C}$, $\nu_j(\varphi_{M_0})=0$ and $\nu_j(\varphi_{F_j})=1$, $F_j$ being the fiber of $\mathbb{F}_\delta$ going through the origin of $\mathcal{C}_{\nu_j}$. Then,
    $$b'(\operatorname{APG}(X,\pi))=\max\Bigg\lbrace \Bigg\lceil\sum_{h=1}^s \# (\mathcal{C}_{\nu_{j_h}})-2\Bigg\rceil^+ \ \bigg| \  {\footnotesize \begin{array}{l}
       \text{$\big\lbrace \mathcal{C}_{\nu_{j_h}} \big\rbrace_{h=1}^s$ is a set of pairwise}\\ \text{disjoint chains whose origins} \\ \text{belong to different fibers}\\
    \end{array}}\Bigg\rbrace,$$
    where $\lceil x \rceil^+$ is defined as the ceil of a rational number $x$ if $x\geq 0$, and $0$ otherwise.

\item[(c)] Suppose that the configuration $\mathcal{C}$ is constellation. Then,
    $$b'(\operatorname{APG}(X,\pi))=\max\Bigg\lbrace \Bigg\lceil\dfrac{\sum_{\lambda=1}^{n_{j}}\mult_{p_{j\,\lambda}}(\varphi_{j\,n_{j}})-2(a_{j\,n_j}+b_{j\,n_j})}{b_{j\,n_j}}\Bigg\rceil^+ \ \bigg| \  1\leq j\leq m
\Bigg\rbrace.
     $$
\end{itemize}
\end{corollary}
\begin{proof}
It follows reasoning as we did in the proof of Corollary \ref{cor:delta0}, but considering Equality \eqref{Eq_proposition_bX} and Lemma \ref{lemma_anticanonical_div}.
\end{proof}

\begin{remark}
Statement (b) also follows from considering a reduced set of divisorial valuations as described in Remark \ref{re:calculo_a_apartado_b_corolario}.
\end{remark}

\subsection{The main result}
In this subsection we state and prove the  \emph{main result in the paper}, Theorem \ref{Thm:Cone_equiv_conds}. It shows that the arrowed proximity graph of a rational surface $X$ obtained from a relatively minimal model which is a Hirzebruch surface $\mathbb{F}_\delta$ determines two values that, when  $\delta$ exceeds or equals one of them, in the first case, the cone $\operatorname{NE}(X)$ is polyhedral and minimally generated and, in the second one, $X$ is a Mori dream space.

We conclude this subsection with a close result to Theorem \ref{Thm:Cone_equiv_conds} for some surfaces obtained by a composition of point blowups over the projective plane.

\begin{theorem}\label{Thm:Cone_equiv_conds}
Let $X$ be a rational surface obtained by a composition $\pi:X\to\mathbb{F}_\delta$ of point blowups given by a configuration of infinitely near points over a Hirzebruch surface $\mathbb{F}_\delta$. Then, there exist two positive integers only depending on the arrowed proximity graph, the value $a(\operatorname{APG}(X,\pi))$ introduced in Proposition \ref{thm:nonnegative_selfinters} and the value $b(\operatorname{APG}(X,\pi)):=\max\{a(\operatorname{APG}(X,\pi)),b'(\operatorname{APG}(X,\pi))\},$ where $b'(\operatorname{APG}(X,\pi))$ is the value introduced in Proposition \ref{prop:K_Xposi_inter},  such that:
\begin{itemize}

\item[1)] If $\delta\geq a(\operatorname{APG}(X,\pi))$, then
\begin{itemize}
\item[(a)] The cone of curves $\operatorname{NE}(X)$ is finite polyhedral and its extremal rays are given by the classes of the strict transforms on $X$ of the fibers on $\mathbb{F}_\delta$ going through the blown up points in $\mathbb{F}_\delta,$ the special section, and the exceptional divisors of $\pi.$
\item[(b)] The nef cone $\operatorname{Nef}(X)$ is generated by the classes of the specific divisors shown in Lemma \ref{lema:Dual_cone_S}.
\end{itemize}
\item[2)] If $\delta\geq b(\operatorname{APG}(X,\pi))$, then $X$ is a Mori dream space.
\end{itemize}
\end{theorem}

\begin{proof}

First, we prove Item 1) (a). Let $H$ be an ample divisor on $X$ and let us define the convex cone
$$
    Q(X):=\lbrace [D]\in \NS_{\mathbb{R}}(X) \ | \ [D]^2\geq 0 \text{ and } [H]\cdot [D]\geq 0\rbrace .
$$
By \cite[Lemma 1.20]{KolMor}, it holds that $Q(X)\subseteq \overline{\NE}(X)$. Furthermore, the convex cone $\mathfrak{C}^\vee(X)$ introduced before Lemma \ref{lema:Dual_cone_S} is embedded in $Q(X)$. This is because every element in $\mathfrak{C}^\vee(X)$ has nonnegative self-intersection whenever $\delta\geq a(\operatorname{APG}(X,\pi))$ which holds by hypothesis. As a consequence, one obtains that $\mathfrak{C}(X)^\vee\subseteq \overline{NE}(X)$. We claim that $\overline{\NE}(X)=\mathfrak{C}(X) + \mathfrak{C}(X)^\vee$, where $+$ means the  Minkowski sum. It is immediate to check that $\overline{\NE}(X)$ contains $\mathfrak{C}(X) + \mathfrak{C}(X)^\vee$. Let us see that $\NE(X)\subseteq \mathfrak{C}(X) + \mathfrak{C}(X)^\vee$. Indeed, consider any reduced and irreducible curve $C$ whose class does not belong to $\mathfrak{C}(X)$. If $[C]\in \mathfrak{C}^\vee(X)$, the inclusion is clear. Otherwise, there is another  curve $C'\neq C$ such that  $[C']\in \mathfrak{C}(X)$ and $[C]\cdot [C']<0,$ which implies that $C$ and $C'$ have a component in common, leading to a contradiction. As a result, the chain of inclusions $\NE(X)\subseteq \mathfrak{C}(X) + \mathfrak{C}(X)^\vee\subseteq \overline{\NE}(X)$ holds. Taking closures and since $\mathfrak{C}(X)$ and $\mathfrak{C}(X)^\vee$ are closed convex cones, we obtain the claimed statement. To finish, we assert that $\mathfrak{C}(X)^\vee \subseteq Q(X) \subseteq (\mathfrak{C}(X)^\vee)^\vee =\mathfrak{C}(X)$ \emph{which proves Item 1) (a)}. We have showed the first inclusion, and the second one follows from \cite[Proposition 2.1]{GalMonMorMoy}.

The facts that $\NE(X)$ is a polyhedral convex cone (by Item 1) (a)) and $\Nef(X)$ is the dual convex cone of $\overline{\NE}(X)$ \emph{show Item 1) (b)}.

\medskip

Finally, we prove \emph{Item 2)}. Recall that a divisor $D$ on $X$ is {\it big} if the rational map $\phi_m: X \dashrightarrow \mathbb{P}H^0(X,\mathcal{O}_X(mD))$ defined by $\mathcal{O}_X(mD)$ is birational onto its image for some $m>0$ (see \cite[Lemma 2.2.3]{Laz1} for an equivalent condition). By \cite[Remark 2.2.27]{Laz1}, the pseudoeffective divisors of $X$ are those with nonnegative intersection product with eve\-ry nef divisor on $X$. In addition, these divisors generate $\overline{\operatorname{NE}}(X)$ (which coincides with $\operatorname{NE}(X)$ in this case). The definition of $b(\operatorname{APG}(X,\pi))$ and Proposition \ref{prop:K_Xposi_inter} show that the intersection product of $-K_X$ with every generator of $\operatorname{Nef}(X)$ is positive. Even more, its class belongs to the interior of $\overline{\operatorname{NE}}(X)$, since otherwise, the intersection product with a generator of $\operatorname{Nef}(X)$ vanishes. As a consequence, by  \cite[Theorem 2.2.26]{Laz1}, $-K_X$ is a big divisor and Item 2) follows from  \cite[Theorem 1]{TesVarVel}.
\end{proof}

\begin{remark}
\label{Paco2}
Theorem \ref{Thm:Cone_equiv_conds} provides information on several of the commonly used cones to study rational surfaces. Indeed, with the above notation, if $\delta\geq a({\rm APG}(X,\pi))$, not only the cones ${\rm NE}(X)$ and ${\rm Nef}(X)$ are explicitly described (from the arrowed proximity graph of the pair $(X,\pi)$), but also the big cone, because it is the interior of the (finitely generated) cone ${\rm NE}(X)$ \cite[Theorem 2.2.26]{Laz1} (whose extremal rays we determine). Moreover, $\delta\geq b({\rm APG}(X,\pi))$ implies that the semiample cone of $X$ (that is, the convex cone spanned by the semiample divisors) coincides with ${\rm Nef}(X)$ \cite{HuKe}.
\end{remark}

\begin{remark}\label{re:reference_Mustapha}
To decide whether the cone of curves (respectively, the Cox ring) of a rational surface $X$ is finite polyhedral (respectively, finite generated) is a difficult issue, \cite{TesVarVel} being an important contribution. The case when $\mathbb{F}_\delta$ is a relatively minimal model has recently been studied in  \cite{mus,FriasLahy2,RosaFrisMus2020,FriasLahy1,AndRosaMus,RosaFrisMus2023}. In these last papers, the authors consider very specific configurations over $\mathbb{F}_\delta,$ mostly with proper points, to study the effective monoid and the Cox ring of $X$. Theorem \ref{Thm:Cone_equiv_conds} in this paper considers a much more general situation because we contemplate any configuration over $\mathbb{F}_\delta$; our results include some cases in the mentioned references. However, our bounds imply that $\operatorname{NE}(X)$ is not only finitely generated but also minimally generated, and therefore our conclusion is more restrictive and happens in a more general context than  in the above citation package.
\end{remark}

To finish this subsection we deduce a consequence of Theorem \ref{Thm:Cone_equiv_conds}. It gives information on the cone of curves NE$(X)$ and the finite generation of the ring Cox$(X)$ for some surfaces $X$ obtained by blowing up a configuration over the projective plane.

Let $X_1$ be the rational surface obtained by blowing up the projective plane $\mathbb{P}^2$ at a closed point $q\in\mathbb{P}^2,$ and denote by $E_q$ the created exceptional divisor. By \cite[Chapter IV, Proposition IV.1]{Beau}, $X_1$ is isomorphic to the Hirzebruch surface $\mathbb{F}_1$, where the special section $M_0$ coincides with the exceptional divisor $E_q$ and the fibers of $\mathbb{F}_1$ are the strict transforms of the lines of $\mathbb{P}^2$ passing through $q$.

Now, let us consider a composition $\pi':X\to\mathbb{P}^2$ of point blowups at a configuration of infinitely near points. This sequence can be factored as $\pi':X\xrightarrow{\pi}\mathbb{F}_1\xrightarrow{\pi_{1}}\mathbb{P}^2$. Denote by $p_1$ the point in $\mathbb{P}^2$ where the blowup $\pi_1:\mathbb{F}_1\cong \text{Bl}_{p_1}(\mathbb{P}^2)\to \mathbb{P}^2$ is centered. It holds that the pre-images by $\pi_1$ of the proper points in $\mathbb{P}^2$ different from $p_1$, and the points in the first infinitesimal neighbourhood of $p_1$ of the configuration giving rise to $\pi'$ belong to  some fibers on $\mathbb{F}_1$. Denote these fibers by $F_1,\ldots,F_r$. They can be regarded as projective lines $L_1,\ldots, L_r$ on $\mathbb{P}^2$ such that their strict transforms $\tilde{L}_i$ on $\text{Bl}_{p_1}(\mathbb{P}^2)$ belong to $|\pi_1^*L-E_{p_1}|$, for $1\leq i\leq r$, where $L$ is a general line on $\mathbb{P}^2$. Therefore, we obtain the following result.

\begin{corollary}\label{cor:cono_P2}
Keep the previously introduced notation. Let $X\xrightarrow{\pi}\mathbb{F}_1\xrightarrow{\pi_{1}}\mathbb{P}^2$ the factorization of a sequence $\pi':X\to\mathbb{P}^2$ of blowups over $\mathbb{P}^2$ as in the above paragraph.
Let $a(\operatorname{APG}(X,\pi))$ (respectively, $b(\operatorname{APG}(X,\pi))$) be the values introduced in Theorem \ref{Thm:Cone_equiv_conds} for $\pi$. If $a(\operatorname{APG(X,\pi)}) = 1$, then
the cone of curves $\operatorname{NE}(X)$ of $X$ is generated by the classes of the strict transforms on $X$ of the projective lines $L_{1},\ldots,L_{r}$ and the classes of the strict transforms of the exceptional divisors. Moreover,  if $a(\operatorname{APG(X,\pi)}) = b(\operatorname{APG(X,\pi)})= 1$, then $X$ is a Mori dream space.
\end{corollary}

\begin{remark}
\label{Paco1}
By \cite{NagataM}, the cone ${\rm NE}(X)$ of the surface $X$ obtained by blowing up $\mathbb{P}^2$ at $9$ points in very general position is not finitely generated (and, therefore, $X$ is not a Mori dream space). The surface $X$ can be obtained by a sequence of $8$ blowups $\pi:X\rightarrow \mathbb{F}_1$ because $\mathbb{F}_1$ is isomorphic to the blowup of $\mathbb{P}^2$ at one point. The arrowed proximity graph $\mathcal{P}:={\rm APG}(X,\pi)$ consists of 8 vertices, each of them with an arrow corresponding to the fiber going through it. However, by Corollaries \ref{cor:delta0} and \ref{cor:value_b} and Theorem \ref{Thm:Cone_equiv_conds}, any rational surface $Y$ obtained by a sequence of blowups over $\mathbb{F}_\delta$  with arrowed proximity graph $\mathcal{P}$ is a Mori dream space whenever $\delta\geq 8$.
\end{remark}

\subsection{Examples}

We conclude this paper with three examples that use our results to provide  rational surfaces $X$ which either are Mori dream spaces or the generators of their cones of curves $\NE(X)$ and nef cones $\Nef(X)$ can be computed as we have explained before. The first two examples are a continuation of Examples \ref{New}, and \ref{ex:grafo_proximidad} and \ref{example_1}, respectively. Our third example also proves that the values $a(\operatorname{APG(X,\pi)}) $ and $ b(\operatorname{APG(X,\pi)})$ introduced in Theorem \ref{Thm:Cone_equiv_conds} can be different.

\begin{example}\label{New2}
Let $X$ be the rational surface considered in Example \ref{New} where we have computed a set $\mathcal{G}$ of divisors whose classes generate the cone $\mathfrak{C}^\vee(X)$. By the proof of Proposition \ref{thm:nonnegative_selfinters} and Remark \ref{GetA},
the value $a(\operatorname{APG}(X,\pi))$ is the smallest $\delta$ such that the following self-intersections are nonnegative.
$$
\begin{array}{c}
\Lambda_{1\,1}^2=\delta - 1,\ \Lambda_{1\,2}^2=\delta - 2,\ W_{(1\,2),(2\,1)}^2=\delta-1=W_{(1\,2),(3\,1)}^2,W_{(1\,1),(2\,2)}^2=4\delta-2,\\[2mm]
W_{(1\,2),(2\,2)}^2=4\delta-6,\ W_{(1\,1),(2\,3)}^2=9\delta-3,\ W_{(1\,2),(2\,3)}^2=9\delta-12.
\end{array}
$$
Notice that  the remaining self-intersections of elements in $\mathcal{G}$ do not influence the computation of $\delta$. Thus, one gets that $a(\operatorname{APG}(X,\pi))=2$. By Theorem \ref{Thm:Cone_equiv_conds}, if $\delta\geq 2,$ $\NE(X)$ is generated by the set $\{[\tilde{F}_1],[\tilde{F}_2],[\tilde{M}_0],[E_1],\ldots, [E_{6}]\}$. The classes of the divisors in the set  $\mathcal{G}$ (computed in Example \ref{New}) generate the nef cone $\operatorname{Nef}(X)$.

Finally, we obtain the value $b'(\operatorname{APG}(X,\pi))$. Any anticanonical divisor $-K_X$ is linearly equivalent to $-K_{\mathbb{F}_\delta}^*-E_{1\,1}^*-E_{1\,2}^*-E_{2\,1}^*-E_{2\,2}^*-E_{2\,3}^*-E_{3\,2}^*$, $K_{\mathbb{F}_\delta}$ being a canonical divisor of $\mathbb{F}_\delta$. By the proof of Proposition \ref{prop:K_Xposi_inter} and Remark \ref{GetBp}, the value $b'(\operatorname{APG}(X,\pi))$ is the smallest positive integer $\delta$ such that the intersection product of $-K_X$ with any divisor in $\mathcal{G}$ is positive. Then, it is enough to force the positivity of the following values:
$$
W_{(1\,2),(2\,2)}\cdot(-K_X)=2\delta \text{ and } W_{(1\,2),(2\,3)}\cdot(-K_X)=3\delta.
$$
Therefore, $b'(\operatorname{APG}(X,\pi))=1$ and $b(\operatorname{APG}(X,\pi))=\max\{2,1\}=2$. Thus, if $\delta\geq 2$, $X$ is a Mori dream space.

We complement our example with a brief comment on the case when we regard $\mathbb{F}_\delta$ as a toric variety. Assume that the fibers $F_1$ and $F_2$ and the proper points $p_{11}$ and $p_{21}$ are fixed by the torus action. Then, blowing-up at these two points, the obtained surface has an induced toric structure. Assume also that all the remaining blown up points are invariant by the induced torus action. Then, independently of the value of $\delta$, $X$ is a toric variety and, hence, the cone of curves ${\rm NE}(X)$ is finitely generated by the classes of divisors described in \cite[Theorem 6.3.20]{CoxLiS}. When $\delta\geq a({\rm APG}(X,\pi))=2$, these classes are those given in 1(a) of Theorem \ref{Thm:Cone_equiv_conds}; when $\delta\in \{0,1\}$, we must include the class of the strict transform of the linearly equivalent to $M$ section that passes through $p_{11}$. Moreover, $X$ is a Mori dream space independently of the value of $\delta$ \cite{Cox}.
\end{example}

\begin{example}\label{example_2}
Let $\pi:X\to\mathbb{F}_\delta$ be the sequence of blowups over a Hirzebruch surface $\mathbb{F}_\delta$ considered in Examples \ref{ex:grafo_proximidad} and \ref{example_1}. We showed the arrowed proxi\-mity graph $\operatorname{APG}(X,\pi)$ in Figure \ref{fig_conf_F_delta}. As we explained in Example \ref{example_1}, by using a computer we are able to obtain a set $\mathcal{G}$ of generators of the cone $\mathfrak{C}^\vee(X)$, which have the form $D$ described in Proposition \ref{thm:nonnegative_selfinters}. Then, the formulae given in Remarks \ref{GetA} and \ref{GetBp} allow us to get the values $a(\operatorname{APG}(X,\pi))$ and $b'(\operatorname{APG}(X,\pi))$. These computations show that the values $a(\operatorname{APG}(X,\pi))$ and $b(\operatorname{APG}(X,\pi))$ introduced in Theorem \ref{Thm:Cone_equiv_conds} equal one. Therefore, by Theorem \ref{Thm:Cone_equiv_conds},  if $\delta\geq 1,$ $\NE(X)$ is generated  by the set $\{[\tilde{F}_1],[\tilde{F}_2],[\tilde{F}_3],[\tilde{M}_0],[E_1],\ldots,$ $[E_{20}]\}$ and, moreover, $X$ is a Mori dream space. The classes of the divisors computed in Example \ref{example_1} are some of the ge\-ne\-rators of the nef cone $\operatorname{Nef}(X)$. As we said in Example \ref{example_1}, if $\delta \geq 1$, $\mathcal{G}$ is a set of generators of $\operatorname{Nef}(X)$.

Furthermore, by Corollary \ref{cor:cono_P2}, if $X$ is the surface created by the sequence of blowups  $\pi':X\xrightarrow{\pi}\mathbb{F}_1\xrightarrow{\pi_{1}}\mathbb{P}^2$, then $X$ is a Mori dream space and $\NE(X)$ is generated by the classes of the divisors $\tilde{L}_1,\tilde{L}_2,\tilde{L}_3,E_1,\ldots,E_{21},$ where $L_i,1\leq i\leq 3,$ is the projective line on $\mathbb{P}^2$ whose strict transform on $\mathbb{F}_1$ can be regarded as the fiber $F_i$ on $\mathbb{F}_1$.

\end{example}

 \begin{figure}[h!]
        \centering
        \includegraphics[scale=0.75]{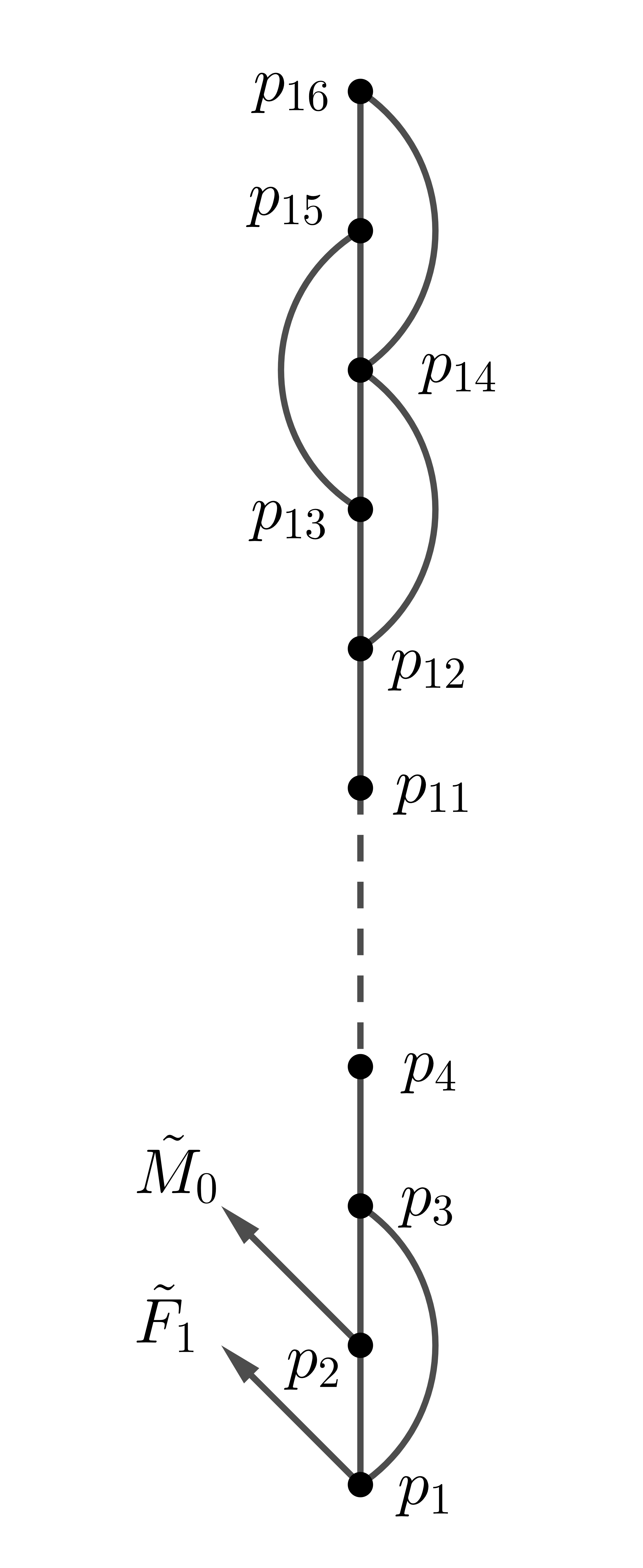}
        \caption{Arrowed proximity graph in Example \ref{ex:a_b_different}.}
        \label{fig:a_b_different}
\end{figure}

\begin{example}\label{ex:a_b_different}
Let $X$ be the rational surface obtained by blowing up a surface $\mathbb{F}_\delta$ at a chain $\mathcal{C}=\{p_l\}_{l=1}^{16}$ of sixteen infinitely near points which corresponds to a unique valuation $\nu_1$. Set $\pi:X\rightarrow \mathbb{F}_\delta$ the composition of blowups given by $\mathcal{C}$. Figure \ref{fig:a_b_different} shows the arrowed proximity graph of the pair $(X,\pi)$. Here, the cone $\mathfrak{C}(X)$ is generated by the set of classes $\{[\tilde{F}_1],[\tilde{M}_0],[E_1],\ldots,[E_{16}]\}$ and, by Lemma \ref{lema:Dual_cone_S},  its dual cone $\mathfrak{C}^\vee(X)$ by the set $\{[F^*],[M^*],$ $[\Lambda_{1}],\ldots,$ $[\Lambda_{16}]\}$, $\Lambda_{k}= \Lambda_{k}(\nu_1)$,  $1 \leq k \leq 16$.

The proof of Corollary \ref{cor:delta0} shows that, to compute the value $a(\operatorname{APG}(X,\pi))$ in Theorem \ref{Thm:Cone_equiv_conds}, it is enough to obtain the divisor $\Lambda_{16}$ and look after the value $\delta$ for which $\Lambda_{16}$ has nonnegative self-intersection.
The proximities equalities and the Noether formula prove that
$$
\begin{array}{c}
a_{16}=(\varphi_{M_0},\varphi_{16})_{p_1}=15, \ \ b_{16}=(\varphi_{F_1},\varphi_{16})_{p_1}=10 \text{ and} \\[2mm]
(\mult_{p_{\lambda}}(\varphi_{16}))_{\lambda=1}^{16}=(10,5,5,5,5,5,5,5,5,5,5,5,3,2,1,1).
\end{array}
$$
Therefore, one gets that
$$
\Lambda_{16}^2=2\,a_{16}\,b_{16} + b_{16}^2\,\delta - \sum_{\lambda=1}^{16}\mult_{p_{\lambda}}(\varphi_{16})^2 = 300 + 100\delta - 390,
$$
and then $a(\operatorname{APG}(X,\pi))$ is the least positive integer $\delta$ satisfying $\Lambda_{16}^2 >0$. Therefore, $a(\operatorname{APG}(X,\pi))=1$.

Now we compute the value $b(\operatorname{APG}(X,\pi))$ introduced in Theorem \ref{Thm:Cone_equiv_conds}.  As before, by Corollary \ref{cor:value_b}, we only need to know the value of $\delta$ guaranteeing that the product $\Lambda_{16}\cdot (-K_X)$ is positive, $-K_X$ being an anticanonical divisor  of the surface $X$. The divisor $-K_X$ is linearly equivalent to the divisor $-K_{\mathbb{F}_\delta}^*-\sum_{\lambda=1}^{16}E_\lambda^*$.  Then,
$$
\Lambda_{16}\cdot (-K_X)=2\,(a_{16}+b_{16}) + b_{16}\delta - \sum_{\lambda=1}^{16}\mult_{p_{\lambda}}(\varphi_{16})= 50 + 10 \delta -72
$$
and thus $b(\operatorname{APG}(X,\pi))=3$.

Therefore, when $\delta\geq 1,$ the cone of curves $\operatorname{NE}(X)$ and the nef cone $\operatorname{Nef}(X)$ of $X$ satisfy the equalities $\operatorname{NE}(X)=\mathfrak{C}(X)$ and $\operatorname{Nef}(X)=\mathfrak{C}^\vee(X)$, while for $\delta\geq 3$, $X$ is a Mori dream space.
\end{example}

\section*{Acknowledgements}
The authors would like to thank the anonymous reviewers for their helpful comments and feedback, which  improved this manuscript.

The third author would like to thank the Department of Algebra, Analysis, Geometry and Topology, and IMUVa of Valladolid University for the support received when preparing this article.

\bibliographystyle{plain}
\bibliography{BIBLIO_paquete1}

\end{document}